\documentclass{amsart}
\usepackage[oldstylemath,fulloldstylenums]{kpfonts}

\usepackage{fullpage}

\usepackage{amsmath}
\usepackage{amsxtra}
\usepackage{amsthm}

\usepackage{mathrsfs}
\usepackage{tikz-cd}
\usepackage{mathtools}
\usepackage{mathpartir}

\usepackage[bbgreekl]{mathbbol}
\usepackage{stmaryrd}
\usepackage{enumitem} 
\usepackage{xcolor}
\definecolor{darkgreen}{rgb}{0,0.45,0} 
\definecolor{darkred}{rgb}{0.75,0,0}
\definecolor{darkblue}{rgb}{0,0,0.6} 
\usepackage[pdfborder=0,pagebackref,colorlinks,citecolor=darkgreen,linkcolor=darkgreen,urlcolor=darkblue]{hyperref}
\usepackage{comment}

\usetikzlibrary{decorations.markings}
\usetikzlibrary{decorations.pathmorphing}
\tikzset{
  module/.style={
    postaction={decorate},
    decoration={
      markings,
      mark=at position #1 with {\arrow{|}}}},
  module/.default=0.5,
  we/.style=
  { postaction={%
      decorate,
      decoration={
        markings,
        mark=at position #1 with {%
          \node[transform shape, yshift=.2em]{%
            \resizebox{0.5em}{!}{$\sim$}};}}}},
  we/.default=0.5,
  we'/.style=
  { postaction={%
      decorate,
      decoration={
        markings,
        mark=at position #1 with {%
          \node[transform shape, yshift=-.2em, rotate=180]{%
            \resizebox{0.5em}{!}{$\sim$}};}}}},
  we'/.default=0.5,
  iso/.style=
  { postaction={%
      decorate,
      decoration={
        markings,
        mark=at position #1 with {%
          \node[transform shape, yshift=.2em]{%
            \resizebox{0.5em}{!}{$\simeq$}};}}}},
  iso/.default=0.5,
  iso'/.style=
  { postaction={
      decorate,
      decoration={
        markings,
        mark=at position #1 with {%
          \node[transform shape, yshift=-.2em, rotate=180]{%
            \resizebox{0.5em}{!}{$\simeq$}};}}}},
  iso'/.default=0.5,
}

\tikzset{
  proarrow/.style={->, module},
  proequal/.style={-, double, module},
  prodotted/.style={->,dotted, module},
  prodashed/.style={->,dashed, module},
  wearrow/.style={->, we},
  wedashed/.style={->, dashed, we},
  wedotted/.style={->, dotted, we},
  tfibarrow/.style={->>, we=0.45},
  tfibdotted/.style={->>,dotted, we=0.45},
  tfibdashed/.style={->>,dashed, we=0.45},
  tcofarrow/.style={>->, we},
  tcofdashed/.style={>->, dashed, we},
  uwearrow/.style={->, we'},
  uwedashed/.style={->, dashed, we'},
  uwedotted/.style={->, dotted, we'},
  utfibarrow/.style={->>, we'=0.45},
  utfibdotted/.style={->>,dotted, we'=0.45},
  utfibdashed/.style={->>,dashed, we'=0.45},
  utcofarrow/.style={>->, we'},
  isoarrow/.style={->, iso},
  isodashed/.style={->, dashed, iso},
  uisoarrow/.style={->, iso'},
  uisodashed/.style={->, dashed, iso'},
  isocell/.style={=>, iso},
  isocelldashed/.style={=>, dashed, iso},
  uisocell/.style={=>, iso'},
  uisocelldashed/.style={=>, dashed, iso'}
}

\DeclareMathOperator*{\aamalg}{\amalg}

\setlist{}
\setenumerate{leftmargin=*,labelindent=0\parindent}
\setitemize{leftmargin=\parindent}

\newtheorem{thm}{Theorem}[subsection]
\newtheorem{lem}[thm]{Lemma}
\newtheorem{prop}[thm]{Proposition}
\newtheorem{cor}[thm]{Corollary}
\newtheorem{conj}[thm]{Conjecture}

\theoremstyle{definition}
\newtheorem{defn}[thm]{Definition}
\newtheorem{ex}[thm]{Example}
\newtheorem{cons}[thm]{Construction}

\theoremstyle{remark}
\newtheorem{rmk}[thm]{Remark}

\newtheorem{dig}[thm]{Digression}

\newtheorem{apo}[thm]{Apology}
\newtheorem{cav}[thm]{Caveat}

\makeatletter
\@addtoreset{section}{part}
\makeatother  

\makeatletter
\let\c@equation\c@thm
\makeatother
\numberwithin{equation}{subsection}

\newcommand{\op}{\textup{op}}

\newcommand{\id}{\textup{id}}
\newcommand{\ob}{\textup{ob}}
\newcommand{\mor}{\textup{mor}}

\newcommand{\ev}{\textup{ev}}
\newcommand{\proj}{\textup{proj}}
\newcommand{\inj}{\textup{inj}}

\newcommand{\sk}{\textup{sk}}
\newcommand{\hocolim}{\textup{hocolim}}

\newcommand{\cod}{\textup{cod}}
\newcommand{\colim}{\textup{colim}}

\newcommand{\cat}[1]{\textup{\textsf{#1}}}

\newcommand{\EE}{\mathbb{E}}
\newcommand{\FF}{\mathbb{F}}

\newcommand{\LL}{\mathbb{L}}

\newcommand{\NN}{\mathbb{N}}
\newcommand{\PP}{\mathbb{P}}

\newcommand{\RR}{\mathbb{R}}

\DeclareMathAlphabet{\mathbbe}{U}{bbold}{m}{n}

\newcommand{\2}{\mathbbe{2}}
\newcommand{\3}{\mathbbe{3}}
\newcommand{\iso}{\mathbb{I}}

\newcommand{\cC}{\mathcal{C}}
\newcommand{\cD}{\mathcal{D}}
\newcommand{\cE}{\mathcal{E}}
\newcommand{\cF}{\mathcal{F}}

\newcommand{\cM}{\mathcal{M}}
\newcommand{\cN}{\mathcal{N}}

\newcommand{\cS}{\mathcal{S}}

\newcommand{\cW}{\mathcal{W}}

\newcommand{\DDelta}{\mathbbe{\Delta}}

\newcommand{\we}{\mathscr{W}}
\newcommand{\cof}{\mathscr{C}}
\newcommand{\fib}{\mathscr{F}}

\newcommand{\sF}{\mathscr{F}}

\newcommand{\sL}{\mathscr{L}}
\newcommand{\sR}{\mathscr{R}}

\newcommand{\bT}{\mathbf{T}}

\newcommand{\Adj}{\mathcal{A}\cat{dj}}
\newcommand{\Cat}{\mathcal{C}\cat{at}}
\newcommand{\Set}{\mathcal{S}\cat{et}}
\newcommand{\sSet}{\cat{s}\mathcal{S}\cat{et}}

\newcommand{\Map}{\textup{Map}}
\newcommand{\set}{\cat{set}}
\newcommand{\sset}{\cat{sset}}

\font\maljapanese=dmjhira at 2ex 
\def\yo{\textrm{\maljapanese\char"48}}

\makeatletter
\def\makeslashed#1#2#3#4#5{#1{\mathpalette{\sla@{#2}{#3}{#4}}{#5}}}

\def\@mathlower#1#2#3{\setbox0=\hbox{$\m@th#2#3$}\lower#1\ht0\box0}
\def\mathlower#1#2{\mathpalette{\@mathlower{#1}}{#2}}
\makeatother

\newcommand\dhxrightarrow[2][]{%
  \mathrel{\ooalign{$\xrightarrow[#1\mkern4mu]{#2\mkern4mu}$\cr%
  \hidewidth$\rightarrow\mkern4mu$}}
}

\newcommand\tailxrightarrow[2][]{%
  \mathrel{\ooalign{$\xrightarrow[#1\mkern4mu]{#2\mkern4mu}$\cr%
  \hidewidth$\Yright\mkern14mu$}}
}

\newcommand{\fto}{\twoheadrightarrow}
\newcommand{\cto}{\rightarrowtail}
\newcommand{\wto}{\xrightarrow{{\smash{\mathlower{0.8}{\sim}}}}}
\newcommand{\cwto}{\tailxrightarrow{{\smash{\mathlower{0.8}{\sim}}}}}
\newcommand{\fwto}{\dhxrightarrow{{\smash{\mathlower{0.8}{\sim}}}}}

\newcommand{\PSH}{\mathcal{PSH}}

\newcommand{\GPD}{\mathcal{GPD}}
\newcommand{\CAT}{\mathcal{CAT}}

\newcommand{\type}[1]{{\textup{#1}}}
\newcommand{\univ}{\mathcal{U}}
\newcommand{\app}{\textup{app}}

\newcommand{\name}[1]{{\ulcorner{#1}\urcorner}}
\newcommand{\Id}{\textup{Id}}
\newcommand{\refl}{\textup{refl}}
\newcommand{\judgment}{\mathcal{J}}

\newcommand{\Eq}{\mathsf{Eq}}
\newcommand{\Equiv}{\mathsf{Equiv}}
\newcommand{\isEquiv}{\mathsf{isEquiv}}
\newcommand{\isContr}{\mathsf{isContr}}
\newcommand{\map}{\mathsf{map}}

\newcommand{\el}[1]{{\int\!{#1}}}

\begin{document}

\title{On the \texorpdfstring{$\infty$}{infinity}-topos semantics of homotopy type theory}
\author{Emily Riehl}
\date{ \emph{Logique et structures sup\'{e}rieures}, CIRM - Luminy, 21-25 February 2022.}

\thanks{I am grateful to Dimitri Ara, Thierry Coquand, and Samuel Mimram for organizing the workshop that provided the occasion for this lecture series. Specific thanks go to Nathanael Arkor, Steve Awodey, Denis-Charles Cisinski, Daniel Gratzer,  Peter LeFanu Lumsdaine, Kenji Maillard, David Michael Roberts, Mike Shulman, Jon Sterling, Andrew Swan for useful discussions that assisted with the preparation of these notes. This material is based upon work supported by the Swedish Research Council under grant no.~2016-06596 while the author was in residence at Institut Mittag-Leffler in Djursholm, Sweden during January 2022. The author is also grateful to receive support from the National Science Foundation via the grants DMS-1652600 and DMS-2204304 and from the United States Air Force Office of Scientific Research under award number FA9550-21-1-0009. Various expository improvements were noticed during the preparation of a talk on ``Homotopy types as homotopy types'' for A panorama of homotopy theory: a conference in honour of Mike Hopkins. The revised version of the manuscript benefited from the perspicuous suggestions of an anonymous referee.}

\address{Department of Mathematics\\Johns Hopkins University \\ 3400 N Charles Street \\ Baltimore, MD 21218}
\email{eriehl@jhu.edu}

\begin{abstract} 
  Many introductions to \emph{homotopy type theory} and the \emph{univalence axiom} gloss over the semantics of this new formal system in traditional set-based foundations. This expository article, written as lecture notes to accompany a 3-part mini course delivered at the Logic and Higher Structures workshop at CIRM-Luminy,  attempt to survey the state of the art, first presenting Voevodsky's simplicial model of univalent foundations and then touring Shulman's vast generalization, which provides an interpretation of homotopy type theory with strict univalent universes in any $\infty$-topos. As we will explain, this achievement was the product of a community effort to abstract and streamline the original arguments as well as develop new lines of reasoning.
\end{abstract}

\maketitle

\setcounter{tocdepth}{2}
\tableofcontents

\section{Overview}

Homotopy type theory is a new field of mathematics that grew out of recently-discovered connections between type theory, homotopy theory, and higher category theory. As a type-theory, it is a stand-alone foundational formal system that can be used to construct mathematical objects and prove theorems about them.  While classical mathematical foundations has two primitive notions --- \emph{propositions}, which are governed by the rules of predicate logic, and \emph{sets}, which satisfy the Zermelo--Frankel axioms --- in homotopy type theory these are unified by a common primitive notion of \emph{type}.\footnote{In Voevodsky's hierarchy of types, the propositions correspond to $-1$-types while the sets correspond to the 0-types, two of the bottom layers of an unbounded infinite hierachy determined by the complexity of a type's \emph{identity types}.} These foundations are \emph{univalent}, validating the common mathematical practice of identifying objects that are equivalent in the sense appropriate to their context. Homotopy type theory is not merely a foundation for mathematics but can also be treated as a powerful programming language, and in particular is amenable to computer formalization, which is an original motivation of one of the pioneers in the field \cite{voevodsky-origins}.

While there are a number of excellent expository introductions to homotopy type theory and univalent foundations, such as the textbooks \cite{HoTT,rijke} and the expository essays \cite{apw, pelayo-warren, shulman-synthetic, shulman-logic}, these sources say relatively little about the interpretation of this formal system in traditional foundations. The aim of this article is to survey the state of the art, as these semantics establish the relative consistency of this formal system, provide rigorous justification for the homotopical intuitions, and connect theorems proven in homotopy type theory with the rest of mathematics. 

\subsection{On univalent foundations and their semantics}

Homotopy type theory is based on Martin-L\"{o}f's dependent type theory \cite{martin-lof}, a specific form of type theory developed to support constructive mathematics. Martin-L\"{o}f's type theory has two forms, referred to as \emph{intensional} and \emph{extensional}, according to the absence or presence of the \emph{equality reflection rule}, which derives a judgmental equality of terms from a term of the corresponding identity type; see Digression \ref{dig:ext-vs-int}.\footnote{The adjectives ``intensional'' and ``extensional'' describe the behavior of the judgmental equality relation of the type system. Identity types are always extensional, in the sense in which this term is used elsewhere in logic and philosophy, so in extensional type theory, when the relation presented by the identity types is equivalent to the judgmental equality relation, judgmental equality is extensional as well. Unfortunately, within the literature it is common to use the phrase ``intensional identity types'' to refer to identity types without the equality reflection rule and ``extensional identity types'' to refer to identity types with the equality reflection rule, and we follow this convention here.} In its original intensional form, Hofmann and Streicher \cite{HS} and later Gambino and Garner \cite{GG} and Awodey and Warren \cite{AW} discovered that identity types admit ``higher-dimensional'' interpretations, in which the identity relation need not be a mere proposition but may carry additional structure. Building on this observation, van den Berg and Garner \cite{vdBG} and Lumsdaine \cite{lumsdaine} showed that the iterated identity types equip the types of Martin-L\"{o}f's type theory with the structure of weak $\omega$-categories (aka $\infty$-groupoids). 

Contemporaneously, Voevodsky discovered homotopical interpretations for the other type constructors of dependent type theory that were compatible with the interpretation of dependent types as \emph{fibrations}. This lead to the discovery of the \emph{univalence axiom}, which characterized the identity types in the universe of types: univalence asserts that identifications between types are equivalent to equivalences via a canonically-defined map. Since all of the constructions in dependent type theory are invariant under identifications, it follows from univalence that all constructions are invariant under equivalence between types. Here we use ``homotopy type theory'' and ``univalent foundations'' as synonyms for Martin-L\"{o}f's dependent type theory \cite{martin-lof}, with intensional identity types, together with Voevodsky's univalence axiom.\footnote{The books \cite{HoTT,rijke} both include \emph{higher inductive types} within the purview of homotopy type theory. The $\infty$-topos semantics of homotopy type theory also provides semantics for higher inductive types, which we do not discuss here; \cite{lumsdaine-shulman, shulman}.} 

This survey article is divided into three parts. In the first part, we describe the general categorical semantics of dependent type theory. In \S\ref{sec:introduction}, we introduce the syntax for Martin-L\"{o}f's dependent type theory and describe the rules for identity types, which we connect to homotopical lifting properties. In \S\ref{sec:cat-semantics}, we define the \emph{category of contexts} as \emph{contextual category} and describe the additional logical structures required to interpret the various type-forming operations. These seem to present insurmountable coherence problems, which can be solved by Voevodsky's technique of building contextual categories from \emph{universes}.

In the second part, we describe the simplicial model of univalent foundations discovered by Voevodsky \cite{klv}, which justifies the relative consistency of the univalence axiom and the homotopical interpretation of type theory. In \S\ref{sec:universe}, we build a universe in the category of simplicial sets, using a construction of Hofmann and Streicher \cite{HS-lifting}. In \S\ref{sec:simplicial-model}, we extract a universal Kan fibration from the Hofmann--Streicher universe and prove its fibrancy and univalence. These results combine to give an interpretation of homotopy type theory with univalent universes in the $\infty$-topos of spaces. We present relatively few of the original arguments of \cite{klv} in this section, instead substituting streamlined lines of reasoning developed by the homotopy type theory, which are more amenable to generalization. 

In the third part, we describe Shulman's \cite{shulman} vast extension of Voevodsky's theorem, establishing an interpretation of homotopy type theory with univalent universes in an arbitrary $\infty$-topos. In \S\ref{sec:topoi}, we introduce $\infty$-topoi via their presentations by model categories that are \emph{model toposes} in the sense of Rezk \cite{rezk}. In \S\ref{sec:topos-semantics}, we then sketch Shulman's proof, a tour de force argument which builds upon all of the results developed up to this point.

There is a good reason why semantics are typically discussed long after an introduction to homotopy type theory: the subject is both technical and has substantial prerequisites. In addition to fluency in ordinary categorical language, an ideal reader of this note would have some exposure to the category of simplicial sets (see \S\ref{ssec:sset}), Quillen model structures (see \S\ref{ssec:quillen}), and dependent type theory (see \S\ref{ssec:dtt}). To help newcomers, or those in need of a refresher, we give a brief introduction to the $\infty$-topos of spaces, as presented by the category of simplicial sets, before diving into the general categorical semantics.

\subsection{Simplicial sets as ``spaces''}\label{ssec:sset}

\begin{apo} Regrettably, it seems to be impossible to introduce the category of simplicial sets in any expedient way. While it is possible to give concise and comprehensive definitions, I'm not convinced that such a terse introduction is actually helpful to someone who is learning the material for the first time who needs to immediately internalize everything that has been introduced in order to keep up with what follows. For readers who have the time, the books \cite[\S I]{GoerssJardine} or \cite[\S VII.5]{MacLane} or the paper \cite{Friedman} are warmly recommended.\footnote{The short note \cite{riehl-sset} was written when the author was first learning this material for herself.}

  Instead, I plan to treat simplicial sets as a black box (to be explored at more leisure at another time with aid of any of the standard references) and will instead point out what we need to know about them:
  \begin{itemize}
    \item Simplicial sets were introduced to be combinatorial models of topological spaces.
    \item The category of simplicial sets is considerably better behaved then even a convenient category of topological spaces. It is a category of presheaves on a small category $\DDelta$ (of finite non-empty ordinals $[0] \coloneqq \{0\}, [1]\coloneqq \{0,1\}, \ldots$ and order-preserving maps) and as such is a Grothendieck 1-topos, and in particular is locally presentable and locally cartesian closed.
    \item By the coYoneda lemma, each simplicial set is canonically a colimit of the standard simplices $\Delta^0, \Delta^1, \ldots$ these being the representable simplicial sets: $\Delta^n \coloneqq \DDelta(-,[n])$.
  \end{itemize}
\end{apo}

A \textbf{simplicial set} $B$ is a presheaf on the simplex category $\DDelta$. It is conventional to write $B_n$ for the set of $n$-\textbf{simplices} in $B$, the value of the functor $B$ at the object $[n] \in \DDelta$. By the Yoneda lemma, each $b \in B_n$ may be encoded by a map $b \colon \Delta^n \to B$ of simplicial sets. A simplicial set can be thought of as some sort of simplicial complex, built inductively by attaching $n$-simplices along their boundary $\partial\Delta^n \subset \Delta^n$, which is built from simplices in lower dimensions. More generally, any monomorphism $i \colon A \cto B$ of simplicial sets factors canonically as
\[
    \begin{tikzcd}[column sep=2em]
        A \arrow[r, tail] & A \cup_{\sk_0A} \sk_0B \arrow[r, tail] & A \cup_{\sk_1A}\sk_1B \arrow[r, tail] & \cdots \arrow[r, tail] & \colim_n (A \cup_{\sk_n A} \sk_n B) \cong B
    \end{tikzcd}
\]
where the inclusion $A \cup_{\sk_{n-1}A}\sk_{n-1}B \cto A \cup_{\sk_nA}\sk_nB$ is defined as a pushout of a coproduct of maps $\partial\Delta^n \cto \Delta^n$, each of which attaches a \emph{non-degenerate} $n$-simplex of $B$ not already present in $A$ along its boundary.\footnote{This requires the law of excluded middle; here and elsewhere we make use of non-constructive principles.} 

There is a class of ``very surjective'' maps between simplicial sets called \textbf{trivial fibrations} denoted $f \colon X \fwto Y$ that are defined by a right lifting property against the monomorphisms:
\[
    \begin{tikzcd} A \arrow[r, "x"] \arrow[d, "i"', tail] & X \arrow[d, utfibarrow, "f"] & \arrow[d, phantom, "\leftrightsquigarrow"] & \partial\Delta^n \arrow[d, tail] \arrow[r, "x"] & X \arrow[d, utfibarrow, "f"] \\ B \arrow[r, "y"'] \arrow[ur, dashed, "\exists"'] & Y & ~ & \Delta^n \arrow[r, "y"'] \arrow[ur, dashed, "\exists"'] & Y
    \end{tikzcd}
\]
Equivalently, since monomorphisms decompose as ``relative cell complexes''---sequential composites of pushouts of coproducts of the maps $\partial\Delta^n\cto\Delta^n$''---the trivial fibrations are characterized by a right lifting property against the simplex boundary inclusions for all $n \geq 0$.

There is something a bit strange about using simplicial sets as a combinatorial model for spaces, namely that the standard simplices are ``directed''; for instance there is no ``reversal'' map $\rho \colon \Delta^1 \to \Delta^1$ that exchanges the two vertices. While maps $p \colon \Delta^1 \to X$ are thought of as ``paths'' in $X$ they only behave like we would expect---meaning paths can be reversed and composed---if $X$ is a special sort of simplicial set called a \textbf{Kan complex}, which is also characterized by a right lifting property:
\[
    \begin{tikzcd} \Lambda^n_k \arrow[d, utcofarrow] \arrow[r, "x"] & X & & \Lambda^n_k \arrow[d, utcofarrow] \arrow[r, "x"] & X \arrow[d, two heads, "f"] \\ \Delta^n \arrow[ur, dashed, "\exists"'] & & & \Delta^n \arrow[ur, dashed, "\exists"'] \arrow[r, "y"'] & Y
    \end{tikzcd}
\]
More generally, a map of simplicial sets is called a \textbf{Kan fibration} and denoted $f \colon X \fto Y$ if it has the right lifting property against all \textbf{horn inclusions} $n \geq 1$, $0 \leq k \leq n$, the \textbf{horn} $\Lambda^n_k \subset \Delta^n$ being the simplicial subset spanned by all of the codimension 1 simplices except the one opposite vertex $k$.

\subsection{The simplicial model for the homotopy theory of homotopy theory}\label{ssec:quillen}

The classes of maps just introduced provide the category of simplicial sets with a \textbf{model structure} in the sense of Quillen \cite{quillen}. To explain what this means, we require a few definitions.

\begin{defn}\label{defn:wfs} A \textbf{weak factorization system} on a category $\cE$ is given by a pair of classes of maps $(\sL, \sR)$ so that:
    \begin{enumerate}
        \item Every morphism in $\cE$ factors as:
        \[\begin{tikzcd}[row sep=small]
            X \arrow[rr, "f"] \arrow[dr, dashed, "\sL \ni\ell"'] & & Y \\ & E \arrow[ur, dashed, "r \in \sR"']
        \end{tikzcd}
        \]
        \item The classes $\sL$ and $\sR$ are characterized by the left and right lifting properties against one another:
        \[\begin{tikzcd}
            A \arrow[d, "\sL \ni \ell"'] \arrow[r, "x"] & X \arrow[d, "r \in \sR"] \\ B \arrow[r, "y"'] \arrow[ur, dashed, "\exists"'] & Y
        \end{tikzcd}\]
        This means that there exist solutions to lifting problems as above and moreover any map with the left lifting property against $\sR$ lies in $\sL$ and any map with the right lifting property against $\sL$ lies in $\sR$.
    \end{enumerate}
\end{defn}

\begin{ex} For example, the category of simplicial sets has the following weak factorization systems:
    \begin{itemize} 
        \item One whose left class is the monomorphisms and whose right class is the trivial fibrations.
        \item Another whose right class is the fibrations and whose left class is called the \textbf{trivial cofibrations}, for reasons that will become clear below.
    \end{itemize}
Both of these weak factorization systems are \textbf{cofibrantly generated} meaning that there is a set of maps that
\begin{itemize}
    \item characterizes the right class, by the right lifting property, and
    \item generates the left class, by taking coproducts, pushouts, transfinite composites and retracts.
\end{itemize}
Here the monomorphisms and trivial cofibrations are generated respectively by the sets
\[ \{ \partial\Delta^n \cto \Delta^n\}_{n \geq 0} \qquad \{\Lambda^n_k \cwto \Delta^n\}_{n \geq 1, 0 \leq k \leq n}.\]
\end{ex}

As model structure is defined by a pair of interacting weak factorization systems as follows:

\begin{defn}\label{defn:model-structure} A \textbf{model structure}\footnote{This is not Quillen's original definition \cite{quillen}, but rather a simplification observed by Joyal and Tierney \cite[7.7]{JT}.} on a complete and cocomplete category $\cE$ is given by three classes of maps $(\cof,\we, \fib)$ --- the \textbf{cofibrations} ``$\cto$'', the \textbf{weak equivalences} ``$\wto$'', and \textbf{fibrations} ``$\fto$'' --- so that
    \begin{enumerate}
        \item $(\cof,\fib\cap\we)$ and $(\cof\cap\we,\fib)$ define a pair of weak factorizations systems, and
        \item the weak equivalences satisfy the 2-of-3 property: if any two of three of $g$, $f$, and $g \cdot f$ are weak equivalences, so is the third.
    \end{enumerate}
\end{defn}

The classes $\cC \cap \cW$ and $\cF\cap\cW$ are called \textbf{trivial cofibrations} ``$\cwto$'' and \textbf{trivial fibrations} ``$\fwto$'' respectively. An object in $\cE$ is \textbf{cofibrant} just when the map to it from the initial object is a cofibration and \textbf{fibrant} just when the map from it to the terminal object is a fibration.

\begin{ex} Quillen established a model structure on the category of simplicial sets whose cofibrations are the monomorphisms, fibrations are the Kan fibrations, and weak equivalences are the \textbf{weak homotopy equivalences}, maps of simplicial sets whose geometric realizations define weakly homotopically equivalent spaces. All objects are cofibrant, while the fibrant objects are exactly the Kan complexes, highlighted above. This is commonly referred to as the \textbf{Kan--Quillen model structure} to distinguish it from other model structures borne by the category of simplicial sets.
\end{ex}

\begin{rmk}\label{rmk:sm7}
The category of simplicial sets is cartesian closed and its self-enrichment behaves well with respect to the Kan--Quillen model structure. To explain, observe that for any pair of maps $\ell \colon A \to B$ and $r \colon X \to Y$, we may define the \textbf{Leibniz exponential}, the induced morphism to the pullback in the square
\begin{equation}\label{eq:leibniz}
    \begin{tikzcd}
X^B \arrow[drr, bend left, "X^\ell"] \arrow[ddr, bend right, "r^B"'] \arrow[dr, dashed, "\ell\hat{\pitchfork}r" description] \\ & \bullet \arrow[d, dotted] \arrow[r, dotted] \arrow[dr, phantom, "\lrcorner" very near start] & X^A \arrow[d, "r^A"] \\ & Y^B \arrow[r, "Y^\ell"']& Y^A
    \end{tikzcd}
  \end{equation}
When $\ell$ and $r$ belong to the left and right classes of either of the weak factorization systems in the Kan--Quillen model structure, the induced map
\[ 
    \begin{tikzcd} X^B \arrow[r, "\ell\hat{\pitchfork}r"] & Y^B \times_{Y^A} X^A \end{tikzcd}
\]
is ``very surjective,'' i.e., defines a trivial fibration. This can be interpreted as an enrichment of the lifting property of weak factorization systems. Trivial fibrations in simplicial sets necessarily have a section. On vertices, this section provides a function that specifies a solution to any lifting problem formed by a commutative square from $\ell$ to $r$, but since this function extends to a morphism of simplicial sets it can be thought of as defining a ``continuous choice of lifts.''

In addition, if $\ell$ is a cofibration and $r$ is a fibration this map is a fibration, supplying a further stability property to the class of Kan fibrations than guaranteed by the unenriched lifting property.
\end{rmk}

These properties make the Kan--Quillen model structure into a \emph{simplicial model category}. 

\begin{defn}\label{defn:simp-model-cat} 
  A category $\cE$ that is enriched, tensored, and cotensored over simplicial sets is a \textbf{simplicial model category} when it admits a model structure and its simplicial cotensor interacts with the Kan--Quillen model structure on simplicial sets in the way described here: for any cofibration $\ell \colon A \cto B$ of simplicial sets and fibration $r \colon X \fto Y$ in $\cE$, the \textbf{Leibniz cotensor} \eqref{eq:leibniz} is a fibration in $\cE$ that is a weak equivalence if either $\ell$ or $r$ are weak equivalences.
\end{defn}

\part*{Lecture I: A categorical semantics of dependent type theory}

\section{Introduction}\label{sec:introduction}

\subsection{Why simplicial sets might model homotopy type theory}\label{ssec:why-simplicial-sets}

For any Kan complex $A$, there is a natural \textbf{path space factorization} of the diagonal map defined by exponentiation with the simplicial interval
\[ 
    \begin{tikzcd} \Delta^0+\Delta^0 \arrow[r, tail] & \Delta^1 \arrow[r, we] & \Delta^0 & \rightsquigarrow & A \arrow[r, utcofarrow, "r"] & A^{\Delta^1} \arrow[r, two heads, "{(e_0,e_1)}"] & A \times A\,.
    \end{tikzcd}
\]
The right map, which evaluates a path at its endpoints, is a Kan fibration, while the left map, the inclusion of constant paths, is a trivial cofibration in Quillen's model structure, and as such has the right lifting property with respect to an arbitrary Kan fibration:
\[
    \begin{tikzcd} A \arrow[d, tcofarrow, "r"'] \arrow[r, "e"] & E \arrow[d, two heads, "p"] & \arrow[d, phantom, "\leftrightsquigarrow"] & A \arrow[rr, bend left, "e"]\arrow[d, tcofarrow, "r"'] \arrow[r, "d"] & P \arrow[d, two heads] \arrow[r] \arrow[dr, phantom, "\lrcorner" very near start] & E \arrow[d, two heads, "p"] \\ A^{\Delta^1} \arrow[r, "b"'] \arrow[ur, dashed] \arrow[ur, dashed] & B & ~ & A^{\Delta^1} \arrow[r, equals] \arrow[ur, dashed, "J"'] & A^{\Delta^1} \arrow[r, "b"'] & B
    \end{tikzcd}        
\]
In particular, by pulling back along the codomain $b$ it suffices to consider right lifting problems against Kan fibrations over the path space $A^{\Delta^1}$, as displayed in the left-hand square of the rectangle above-right. This is the semantic interpretation of the homotopy type theoretic principle of \textbf{path induction}: given a type family $x : A, y : A, p : x=_A y \vdash P(x,y,p)$ and a family of terms $a :A \vdash d(a) : P(a,a, r_a)$ over the constant paths, there exists a section \[x : A, y : A, p : x=_A y \vdash J_d(x,y,p): P(x,y,p)\] such that $a :A \vdash J_d(a,a,r_a) \equiv d(a) : P(a,a, r_a)$.

The homotopy type theoretic principle of path induction is stronger than this, because $A$ might be a dependent type $\Gamma \vdash A$ in an arbitrary context $\Gamma$. But we can perform an analogous construction starting from an arbitrary Kan fibration $p \colon \Gamma.A \fto \Gamma$. First form the path space factorizations for both $\Gamma.A$ and $\Gamma$
\begin{equation}\label{eq:relative-path-space-factorization}
    \begin{tikzcd}[sep=tiny]
        \Gamma.A \arrow[rrr] \arrow[dddd, two heads] \arrow[dr, tcofdashed] & & & (\Gamma.A)^{\Delta^1} \arrow[rrr] \arrow[dddd, two heads] & & & \Gamma.A \times \Gamma.A \arrow[dddd, two heads] \\ &  \Delta^1 \pitchfork_\Gamma (\Gamma.A) \arrow[dddl, two heads, dotted] \arrow[dddr, phantom, "\lrcorner" very near start]\arrow[urr, dotted] \arrow[dr, two heads, dashed ]
        \\ & & (\Gamma.A) \times_\Gamma (\Gamma.A) \arrow[ddll, two heads, dotted] \arrow[uurrrr, dotted]  \arrow[ddr, phantom, "\lrcorner" very near start]\\ ~
        \\ \Gamma \arrow[rrr] & & ~ & \Gamma^{\Delta^1} \arrow[rrr] & & & \Gamma \times \Gamma
    \end{tikzcd}
  \end{equation}
and then pull back the top factorization so that it lies in the slice over $\Gamma$. This constructs a factorization of the fibered diagonal, in the slice over $\Gamma$, using the cotensor with the simplicial interval in the slice over $\Gamma$. A point in the fibered path space $\Delta^1 \pitchfork_\Gamma (\Gamma.A)$ is a path in $\Gamma.A$ that lies over a constant path in $\Gamma$.  Once more, the fibered endpoint evaluation is a Kan fibration, while the fibered inclusion of constant paths is a trivial cofibration, so we have a lifting property exactly as above.\footnote{This follows the properties axiomatized in Definition \ref{defn:simp-model-cat} and the 2-of-3 property of the weak equivalences. Note if $\Gamma$ is not a complex, the non-fibered path spaces factorizations for $\Gamma$ and $\Gamma.A$ need not involve trivial cofibrations and fibrations.} Moreover, since pullback is a simplicially enriched right adjoint, this construction is stable under pullback along any simplicial map $f \colon \Delta \to \Gamma$.
\[
\begin{tikzcd}[sep=small]  \Delta.f^*A \arrow[dr, tcofarrow] \arrow[ddddrr, phantom, "\lrcorner" pos=.05] \arrow[ddddr, two heads] \arrow[rrr, dotted] & & &[-35pt] \Gamma.A \arrow[dr, tcofarrow] \arrow[ddddr, two heads]\\ & \Delta^1 \pitchfork_\Delta (\Delta.f^*A) \arrow[dr, two heads] \arrow[ddd, two heads] \arrow[rrr,dotted] \arrow[dddr, phantom, "\lrcorner" very near start] & & & \Delta^1 \pitchfork_\Gamma (\Gamma.A) \arrow[dr, two heads] \arrow[ddd, two heads]\\ & & (\Delta.f^*A) \times_\Delta (\Delta.f^*A) \arrow[ddl, two heads] \arrow[rrr, dotted] \arrow[ddr, phantom, "\lrcorner" very near start]& & & (\Gamma.A) \times_\Gamma (\Gamma.A) \arrow[ddl, two heads] \\ ~ \\ & \Delta \arrow[rrr, dotted, "f"'] &~ & ~& \Gamma
\end{tikzcd}    
\]
Type-theoretic constructions are necessarily ``stable under substitution,'' where ``substitutions'' may be thought of as context morphisms; see Definition \ref{defn:cat-of-contexts}. Thus, a central requirement of their categorical interpretations is stability under pullback in a sense to be discussed in more detail in \S\ref{ssec:cc-universe}.

\subsection{Martin-L\"{o}f's Dependent Type Theory}\label{ssec:dtt}

We now introduce the syntax and structure rules of Martin-L\"{o}f's dependent type theory, which was referenced informally above. This type theory has four basic judgment forms:

\[ \Gamma \vdash A \qquad \Gamma \vdash a : A \qquad \Gamma \vdash A \equiv B \qquad \Gamma \vdash a \equiv b : A\]
asserting that $A$ is a \textbf{type}, $a$ is a \textbf{term} of type $A$, and that $A$ and $B$ or $a$ and $b$ are \textbf{judgmentally equal} types or terms, all in context $\Gamma$. Here $\Gamma$ is shorthand for an arbitrary \textbf{context}, given by a sequence of variables of previously-defined types. Some types will be defined in the empty-context
\[ \cdot \vdash \NN \]
Such types then form contexts of length one $[n : \NN]$. A type that is defined in a context of length one
\[ n : \NN \vdash \RR^n \]
can then be used to \textbf{extend} the context to a context of length two $[n : \NN, v : \RR^n]$. Types definable using any subset of the previously-defined variables
\[ n : \NN, v : \RR^n \vdash T_v\RR^n\]
may then extend the context further $[n : \NN, v : \RR^n, \ell : T_v\RR^n]$. Fully elaborated, the first two judgments take the form
\begin{align*} x_1 :A_1, x_2 : A_1(x_1), \ldots, x_n : A_n(x_1,\ldots, x_{n-1}) & \vdash B(x_1,\ldots, x_n) \\
  x_1 :A_1, x_2 : A_1(x_1), \ldots, x_n : A_n(x_1,\ldots, x_{n-1}) & \vdash b(x_1,\ldots, x_n) : B(x_1,\ldots, x_n)
\end{align*}
We may think of such judgments as functions that take as input a sequence of typed variables and produce either a type or a term as output. 

The \textbf{structural rules} of dependent type theory govern the variables and the judgmental equality relations. For instance, the \textbf{substitution rule} takes the form
\[
\inferrule{ \Gamma \vdash a : A \\ \Gamma, x : A, \Delta \vdash \judgment }{\Gamma, \Delta[a/x] \vdash \judgment[a/x]}\]
where $\judgment$ can be the conclusion of any of the four judgment forms. This says that given any term $a : A$ defined in context $\Gamma$ and any judgment whose context involves a variable $x: A$, there is a corresponding judgment that no longer involves a variable $x : A$ but instead substitutes the term $a$ for $x$ whenever that variable occurred. 

The \textbf{variable rule} asserts the existence of projection functions
\[
    \inferrule{~}{\Gamma, x : A, \Delta \vdash x : A}
\]
that take any variable occurring in a context and providing it as a term of the corresponding type.

This base type theory then inherits a richer structure on account of various additional \textbf{logical rules} that can be used to form new types. These are called ``logical rules'' because the types themselves are used to encode mathematical statements, which are ``grammatically correct'' but are not necessarily provable or true. To prove such a statement, one the constructs a term of the corresponding type by iteratively applying rules.

For instance, given a context $\Gamma, x :A \vdash B(x)$ there is a \textbf{dependent pair type}
 and a \textbf{dependent function type} 
 which in this case have the same \textbf{formation rules}:
\[
\inferrule{\Gamma, x :A \vdash B(x)}
{ \Gamma \vdash \Sigma_{x:A}B(x)}
\qquad \qquad
\inferrule{\Gamma, x:A \vdash B(x)}
{\Gamma \vdash \Pi_{x:A}B(x)}
\]
These types are distinguished by their \textbf{introduction rules}, which provide terms:
\[
\inferrule{\Gamma, x : A \vdash B(x) \\ \Gamma \vdash a : A \\ \Gamma \vdash b : B(a)}
{\Gamma \vdash (a,b) : \Sigma_{x:A} B(x)}
\qquad\qquad
\inferrule{\Gamma, x :A \vdash B(x) \\ \Gamma, x :A \vdash b(x) : B(x)}
{\Gamma \vdash \lambda x.b(x) : \Pi_{x:A} B(x)}
\]
 The \textbf{elimination rules} explain how generic terms $z : \Sigma_{x:A} B(x)$ and $f : \Pi_{x:A} B(x)$ can be used, 
while the \textbf{computation rules} describe the composites of the functions encoded by the introduction and elimination rules. See \cite[\S A.2]{HoTT} for more.

\subsection{Identity types}

Of the various logical rules that animate dependent type theory, Martin-L\"{o}f's rules for identity types are of paramount importance.

The \textbf{formation rule}
\[
\inferrule{~}{\Gamma, x : A, y : A \vdash x=_A y} 
\]
says that it is meaningful to inquire whether two terms belonging to a common type may be identifiable. The \textbf{introduction rule}
\[
\inferrule{~}{\Gamma, a : A \vdash \refl_a : a =_A a}
\]
guarantees that any term is automatically identifiable with itself, via ``reflexivity.'' 

The \textbf{elimination rule} provides a version of Leibniz' indiscernability of identicals:
\[
\inferrule{\Gamma, x : A, y : A, p : x =_A y \vdash P(x,y,p) \\\Gamma, a : A \vdash d(a) : P(a,a,\refl_a) }
{\Gamma, x : A, y :A, p : x=_A y \vdash J_d(x,y,p) : P(x,y,p)}
\]
Here $P$ can be be thought of as a predicate involving two terms of the same type and an identification between them, though the power of this rule derives from its applicability to arbitrary type families, not only the mere propositions. In the homotopy theoretic interpretation previewed in \S\ref{ssec:why-simplicial-sets}, the elimination rule for identity types is nicknamed \emph{path induction}. See \cite{HoTT,rijke} for a survey of its myriad consequences.

Finally, the \textbf{computation rule} 
\[
\inferrule{\Gamma, x : A, y : A, p : x =_A y \vdash P(x,y,p) \\\Gamma, a : A \vdash d(a) : P(a,a,\refl_a) }
{\Gamma, a :A \vdash J_d(a,a,\refl_a) \equiv d(a): P(a,a,\refl_a)}
\]
is analogous to the computation rule for a recursively defined function.\footnote{Identity types and the type of natural numbers are both \emph{inductive types}, freely generated by the terms provided by their introduction rules. See \cite[\S I.3-5]{rijke}.}

\begin{dig}[extensionality vs intensionality]\label{dig:ext-vs-int}
By the structural rules for judgmental equality, if 
\[ \Gamma, x :A, y : A \vdash x \equiv y : A\]
then
\[ \Gamma, x : A, y : A \vdash \refl_x : x =_A y\]
Thus judgmentally equal terms are always identifiable.

The \textbf{extensional} form of Martin-L\"{o}f's identity types adds a converse implication:
\[ \Gamma, x :A , y : A , p: x =_A y \vdash x \equiv y : A\]
while the \textbf{intensional} form of Martin-L\"{o}f's identity types does not add this implication. 

While the identity types in homotopy type theory are intensional, the early models of Martin-L\"{o}f's identity types were extensional. The first intensional model was given by Hofmann and Streicher in the category of groupoids \cite{HS}, a breakthrough that marked a key step on the way to the development of homotopy type theory.
\end{dig}

We now explain what it means for a category to model type theory.

\section{A categorical semantics of dependent type theory}\label{sec:cat-semantics}

\subsection{The category of contexts}

To interpret syntactic expressions in type theory as data in a category we must first build a category out of type theoretic syntax. A standard way to do this involves Cartmell's \textbf{contextual categories} \cite{cartmell}.

\begin{defn}\label{defn:cat-of-contexts} Let $\bT$ be a type theory with the standard structural rules: variable, substitution, weakening, and the usual rules concerning judgmental equality of types and terms. Then the \textbf{category of contexts} $\cC(\bT)$ is a category in which:
  \begin{itemize}
    \item objects are contexts $[x_1:A_1,\ldots, x_n:A_n]$ of arbitrary finite length up to definitional equality and renaming of free variables;
    \item maps of $\cC(\bT)$ are \textbf{context morphisms} or \textbf{substitutions}, considered up to definitional equality and renaming of variables: a map 
    \begin{equation}\label{eq:context-morphism} f = [f_1,\ldots,f_n] \colon  [y_1: B_1, \ldots , y_m :B_m] \to [x_1:A_1,\ldots, x_n :A_n]\end{equation} is an equivalence class of term judgments
    \begin{align*} y_1: B_1, \ldots, y_m : B_m &\vdash f_1 : A_1 \\ \vdots \qquad &\vdash \qquad \vdots \\ y_1: B_1, \ldots, y_m : B_m &\vdash f_n : A_n(f_1,\ldots, f_{n-1})
    \end{align*}
    \item composition is given by substitution and the identity is defined by using variables as terms.
  \end{itemize}
The category of contexts has the following additional structures:
  \begin{itemize}
    \item $\cC(\bT)$ has a terminal object  corresponding to the empty context.
    \item The objects are partitioned  according to the length of the context
    \[\ob\,\cC(\bT) = \aamalg_{n :\NN} \ob_n\cC(\bT)\] with the empty context being the only object of length 0.
    \item Each object of positive length has a \textbf{display map} 
    \[ [x_1,\ldots, x_n] \colon [x_1 : A_1, \ldots, x_{n+1}: A_{n+1}] \to [x_1 :A_1,\ldots, x_n : A_n] \]
    namely the canonical projection away from the final type in the context. Writing $\Gamma$ for the codomain context, we abbreviate this display map as 
    \[p_{A_{n+1}} \colon \Gamma.A_{n+1} \to \Gamma\] and refer the domain of any display map over $\Gamma$ as a \textbf{context extension} of $\Gamma$. We refer to composites 
    \[ \begin{tikzcd} \Gamma. A.B \arrow[r, "p_B"] & \Gamma.A \arrow[r, "p_A"] & \Gamma & \text{or} & \Gamma.\Delta \arrow[r, "p"] & \Gamma\end{tikzcd}\] of display maps as \textbf{dependent projections}.
    \item For each substitution \eqref{eq:context-morphism} and any context extension $\Gamma .A$ of its codomain $\Gamma$, there is a \textbf{canonical pullback} defined by the context 
    \[[y_1 : B_1,\ldots , y_m : B_m, y_{m+1} : A(f_1(\vec{y}), \ldots, f_n(\vec{y}))]\] together with a canonical pullback square 
    \[ \begin{tikzcd} {\Delta.f^*A} \arrow[d, "p_{f^*A}"'] \arrow[dr, phantom, "\lrcorner" very near start]\arrow[r, "q(f)"] & {\Gamma.A} \arrow[d, "p_A"] \\ \Delta \arrow[r, "f"'] & \Gamma\end{tikzcd}
    \] 
    Moreover, these canonical pullbacks are \emph{strictly stable under substitution}. Given any context morphism $g \colon \Theta \to \Delta$, $\Theta.g^*f^*A = \Theta.(f\cdot g)^*A$ and $q_f \cdot q_g = q_{f \cdot g}$; similarly $\Gamma.1^*A = \Gamma.A$.\footnote{Note the express \emph{equalities} of objects, not just isomorphisms.}
  \end{itemize}
\end{defn}

A \textbf{contextual category} is a category with the structures just enumerated: a terminal object, the natural number lengths, the specified display maps, and chosen canonical pullbacks that are strictly functorial.

\begin{rmk} A term $\Gamma \vdash a : A$ in a type $A$ in context $\Gamma$ is given by a section to the display map $p_A \colon \Gamma.A \to \Gamma$. Thus, in a contextual category, the phrase ``section'' is typically reserved to mean section of a display map.
\end{rmk}

\subsection{Homotopy theoretic models of identity types}

The categorical semantics of intensional identity types were discovered contemporaneously by Awodey--Warren \cite{AW} and Gambino--Garner \cite{GG}. The title of this section alludes to the title of the former's paper, while our presentation focuses on the latter's result, constructing what Gambino and Garner call the \textbf{identity type weak factorization system} on the category of contexts.

\begin{thm}[{\cite{GG}}] Let $\bT$ be a dependent type theory with intensional identity types. Then the category of contexts $\cC(\bT)$ admits a weak factorization system whose left class is comprised of those maps that lift against the display maps.
\end{thm}

In the proof, Gambino and Garner construct the factorization of an arbitrary context morphism, but the main idea is illustrated effectively by the simplest non-trivial case: factoring the identity map on a context $[x :A]$ of length 1 as
\[ \begin{tikzcd}[row sep=small] {[a : A]} \arrow[rr, equals] \arrow[dr, "{[a,a,\refl_a]}"'] & & {[a : A]}  \\ &{[x : A, y : A, p: x =_A y]} \arrow[ur, "{[y]}"'] 
\end{tikzcd}
\]
The right factor is isomorphic in the category of contexts to a composite of display maps and hence belongs to the right class. To see that the left factor is in the left class, we must show it has the left lifting property with respect to an arbitrary display map $p_P \colon \Phi.P \to \Phi$. By pulling back along the codomain of the lifting problem, we may assume that the codomain of this display map is the context $[x:A, y:A, p: x=_A y]$ presenting us with a lifting problem of the form below:
\[
    \begin{tikzcd} 
         {[a : A]} \arrow[d, "{[a, a, \refl_a]}"'] \arrow[r, "{[x, y, p, d(a)]}"] & {[x : A, y:A, p : x=_A y, z : P(x,y,p)]} \arrow[d, "p_P"] \\ 
         {[x : A, y :A, p:x=_A y]} \arrow[r, equals] \arrow[ur, dashed, "{[x,y,p, J_d(x,y,p)]}"'] & {[x : A, y :A, p:x=_A y]} 
    \end{tikzcd}        
\]
By the elimination rule for identity types, we can use the data provided by the lifting problem to define the term 
\[ x : A, y : A, p: x =_A y \vdash J_d(x,y,p) : P(x,y,p)\]
which defines a section of the display map associated to the type family \[ x :A, y :A, p: x=_A y \vdash P(x,y,p),\]
making the bottom triangle commute. The top triangle commutes by the computation rule.

\begin{rmk} Unusually, this factorization is not \emph{functorial}. Given a commutative rectangle involving the identity functions between three contexts of length one:
    \[
        \begin{tikzcd}
            {[a : A]} \arrow[r, "g"] \arrow[d, equals] & {[b: B]}\arrow[r, "f"] \arrow[d, equals] & {[c:C]} \arrow[d, equals] \\ {[a : A]} \arrow[r, "g"] & {[b: B]}\arrow[r, "f"] & {[c:C]}
        \end{tikzcd}
    \]
functoriality would demand that the induced substitutions between identity types commute. But as the maps induced by $f$, $g$, and $f \cdot g$ are defined by path induction (i.e., by the elimination rule of identity types) they only commute up to pointwise identification, not up to judgmental equality.
\end{rmk}

Despite this failure of functoriality, nevertheless, as these structures on the category of contexts are defined using the rules of type theory, they are strictly stable under substitution. The central components of the identity type weak factorization system equip the contextual category $\cC(\bT)$ with what is called an ``identity-type structure.''

\begin{defn}\label{defn:id-structure} An \textbf{identity type structure} on a contextual category consists of:
\begin{itemize}
\item for each object $\Gamma.A$, an object $\Gamma.A.A.\Id_A$ extending the context defined by the canonical pullback:
\[ 
\begin{tikzcd} \Gamma.A.A \arrow[r] \arrow[d] \arrow[dr, phantom, "\lrcorner" very near start] & \Gamma.A \arrow[d, "p_A"] \\ \Gamma.A \arrow[r, "p_A"'] & \Gamma
\end{tikzcd}
\]
\item for each $\Gamma.A$, a morphism
\[
\begin{tikzcd} & \Gamma.A.A.\Id_A \arrow[d, "p_{\Id_A}"] \\ \Gamma.A \arrow[ur, "\refl_A", dashed] \arrow[r, "{(1,1)}"'] & \Gamma.A.A
\end{tikzcd}
\]
lifting the fibered diagonal; and
\item for each $\Gamma.A.A.\Id_A.P$  and solid-arrow commutative diagram
\[
    \begin{tikzcd} 
         \Gamma.A \arrow[d, "{\refl_A}"'] \arrow[r, "{d}"] & \Gamma.A.A.\Id_A.P \arrow[d, "p_P"] \\ 
\Gamma.A.A.\Id_A \arrow[r, equals] \arrow[ur, dashed, "{J}"'] & \Gamma.A.A.\Id_A
    \end{tikzcd}        
\]
a section $J$ of $p_P$ so that $J \cdot \refl_A = d$
\end{itemize}
so that all this structure is strictly stable under pullback along any substitution $f \colon \Delta \to \Gamma$.
\end{defn}
  
\subsection{Initiality}

A \textbf{contextual functor} is a functor between contextual categories preserving all of the structure on the nose. In particular, if the contextual categories have additional logical structure, such as the identity type structure just defined, or the corresponding structures for $\Sigma$- or $\Pi$-types, then the contextual functor must preserve this as well, up to equality.

The category of contexts is intended to be the universal contextual category built from a type theory in the sense made precise by the following conjecture:

\begin{conj}[initiality conjecture] Let $\bT$ be a type theory given by the standard structural rules together with some collection of the logical rules alluded to above. Then the category of contexts $\cC(\bT)$ is the initial contextual category with the corresponding extra structure. 
\end{conj}

Then further models of $\bT$ in a category $\cE$ can be defined by turning $\cE$ into a contextual category with the appropriate logical structure: by initiality, this would induce a unique functor $\cC(\bT) \to \cE$ providing an interpretation of the syntax of type theory in the category $\cE$.

The difficulties with the initiality conjecture are multifaceted. On the one hand:
\begin{quote}
  The trouble with syntax is that it is very tricky to handle rigorously. Any full presentation must account for (among other complications) variable binding, capture-free substitution, and the possibility of multiple derivations of a judgment; and so any careful construction of an interpretation must deal with all of these, at the same time as tackling the details of the particular model in question. Contextual categories, by contrast, are a purely algebraic notion, with no such subtleties. \cite[p.~2075]{klv}.
\end{quote}
Another challenge is that there are many versions of type theory, determined by various collections of rules. This makes it difficult to give a precise definition of what constitutes a type theory that is sufficiently general to cover all of the desired examples. An early version of initiality was proven for a type theory called the \emph{calculus of constructions} by Streicher \cite{streicher-initiality}. This has often been cited as evidence that initiality should be true for a much broader class of type theories, but at the time the simplicial ``model'' of homotopy type theory was being developed, Voevodsky argued that it was  unacceptably un-rigorous to assume initiality without proof.

Since \cite{klv} was written, there was have been efforts to at least prove initiality for the ``Book-HoTT'' version of homotopy type theory appearing in \cite{HoTT}. This was achieved in a 2020 licentiate thesis of Menno de Boer \cite{de-boer} based on parallel formalization efforts undertaken by  Brunerie and de Boer in Agda and Lumsdaine and M\"{o}rtberg in Coq. 

Alternatively, the question of what constitutes a type theory can be answered in such a way as to render initiality automatic. Many practitioners equate type theories with their corresponding suitably-structured categories, whether these be \emph{contextual categories} \cite{cartmell}, or \emph{comprehension categories} \cite{jacobs}, or \emph{categories with families} \cite{dybjer}, or \emph{natural models} \cite{awodey} or something else. Each of these notions defines an \emph{essentially algebraic theory} and thus the category of all such admits an initial object. From that point of view, the only question is whether a particular syntax presents the initial object, and this may be viewed as a practical consideration more so than a theoretic one as the utility of a particular syntactic presentation derives from its use in proofs.

\subsection{Contextual categories from universes}\label{ssec:cc-universe}

Even assuming initiality, it seems daunting to construct a model of type theory in a category $\cE$, for one must still define a contextual category with suitable structures for all the type forming operations.   By analyzing the type-forming operations, one can predict what sort of category might be suitable. For instance, a $\Pi$-type structure on a contextual category requires an operation that takes a composable pair of display maps as below-left to a display map as below-right:
\[ 
\begin{tikzcd} \Gamma.A.B \arrow[r, "p_B"] & \Gamma.A \arrow[r, "p_A"] & \Gamma & \rightsquigarrow & \Gamma.\Pi_AB \arrow[r, "p_{\Pi_AB}"] & \Gamma
\end{tikzcd}
\]
with an introduction rule that says that sections of $p_{\Pi_AB}$ are given by sections of $p_B$. This, together with the fact that contextual categories must admit pullbacks of display maps, suggests that it would be reasonable to restrict to categories $\cE$ that are \textbf{locally cartesian closed}, meaning that for any $f \colon \Delta \to \Gamma$ the composition functor has two right adjoints
\begin{equation}\label{eq:lcc}
\begin{tikzcd}[sep=large] \cE_{/\Delta} \arrow[r, bend left=35, "\Sigma_f", "\bot"'] \arrow[r, bend right=35, "\Pi_f"', "\bot"] & \cE_{/\Gamma} \arrow[l, "f^*" description]
\end{tikzcd}
\end{equation}
defined by pullback along $f$ and pushforward along $f$, respectively. Note when $\cE$ has a terminal object $1$---as is required to model the empty context---this implies that $\cE \cong \cE_{/1}$ and indeed any slice $\cE_{/\Gamma}$ is cartesian closed and has all finite limits.

Note, however, that the definition of a contextual category, implicitly given in Definition \ref{defn:cat-of-contexts}, requires strict stability of the canonical pullback squares, not just the existence of pullbacks. In addition, each logical structure, such as the identity type structure of  Definition \ref{defn:id-structure}, requires the pullbacks to preserve everything on the nose, up to equality of objects. Thus, the more categorically-natural requirement, that pullbacks preserve the various logical structures up to isomorphism, is not enough.

This leads to a massive coherence problem. It is not generally possible to choose strictly functorial pullbacks in a category with pullbacks, though it is possible to replace a category with pullbacks by an equivalent category with strictly functorial pullbacks.\footnote{This can be achieved by applying the general techniques for replacing a fibration by a split fibration to the codomain-projection fibration $\cod \colon \cE^\2 \to \cE$; see \cite{lumsdaine-mathoverflow} for a discussion.} But even after this is achieved, there remains the task of ensuring strict stability of all the logical structures.

Voevodsky's approach to the coherence problem makes use of a ``universe'' in a category, defined as follows:

\begin{defn}[{\cite[2.1]{voevodsky-C}}]\label{defn:universe}
 A \textbf{universe} in a category $\cE$ consists of an object $U$, together with a morphism $\pi \colon \tilde{U} \to U$, and, for each map $A \colon \Gamma \to U$, a choice of pullback square
 \[ \begin{tikzcd} {(\Gamma; A)} \arrow[d, "p_{A}"'] \arrow[r, "{q_A}"] \arrow[dr, phantom, "\lrcorner" very near start] & \tilde{U} \arrow[d, "\pi"] \\ \Gamma \arrow[r, "A"'] & U
\end{tikzcd}
\]
\end{defn}

Given a pair of maps $A \colon \Gamma \to U$, $B \colon (\Gamma;A) \to U$, we write $(\Gamma; A, B)$ for $((\Gamma;A);B)$, and we extend this abbreviation to finite sequences.

\begin{cons}[{\cite[2.12]{voevodsky-C}}]\label{cons:universe-category}
  Given a category $\cE$ with a universe $U$ and a terminal object $1$, define a contextual category $\cE_U$ as follows:
  \begin{itemize}
    \item The objects of $\cE_U$ are finite sequences of morphisms as follows:
    \[\ob_n\cE_U = \{ (A_1,\ldots, A_n)  \in (\mor\cE)^n \mid A_i \colon (1; A_1,\ldots, A_{i-1}) \to U,  \forall i\}.\]
    \item The morphisms between a pair of objects are defined by 
    \[\cE_U((B_1, \ldots, B_m), (A_1,\ldots, A_n)) \coloneqq \cE( (1; B_1, \ldots B_m),(1;A_1,\ldots, A_n)).\]
    \item The terminal object is the empty sequence of morphisms.
    \item For any object $(A_1,\ldots, A_{n+1})$ of positive length, its display map is provided by the universe structure:   
    \begin{equation}\label{eq:universe-display-map} \begin{tikzcd} {(1;A_1,\ldots, A_{n+1})} \arrow[d, "p_{A_{n+1}}"'] \arrow[r] \arrow[dr, phantom, "\lrcorner" very near start] & \tilde{U} \arrow[d, "\pi"] \\ {(1;A_1,\ldots, A_n)} \arrow[r, "A_{n+1}"'] & U \end{tikzcd} \end{equation}
    \item Finally the canonical pullback associated to a morphism $f \colon (B_1,\ldots, B_m) \to (A_1,\ldots, A_n)$ and a display map \eqref{eq:universe-display-map} is defined by factoring the chosen pullback square for the composite $A_{n+1} \cdot f$ through the pullback that defines the context extension $(A_1,\ldots, A_{n+1})$:
    \[
      \begin{tikzcd}
        {(1;B_1,\ldots, B_m, A_{n+1}\cdot f)} \arrow[r] \arrow[d, "p_{f^*A_{n+1}}"'] \arrow[dr, phantom, "\lrcorner" very near start] & {(1;A_1,\ldots, A_{n}, A_{n+1})} \arrow[d, "p_{A_{n+1}}"'] \arrow[r] \arrow[dr, phantom, "\lrcorner" very near start] & \tilde{U} \arrow[d, "\pi"] \\ {(1; B_1,\ldots, B_m)} \arrow[r, "f"'] & {(1;A_1,\ldots,A_n)} \arrow[r, "A_{n+1}"'] & U
      \end{tikzcd}
    \]
  \end{itemize}
\end{cons}

Indeed every small contextual category arises this way:

\begin{prop}[{\cite[5.2]{voevodsky-C}}]
  Let $\cC$ be a small contextual category. Consider the universe $U$ in the presheaf category $[\cC^\op, \Set]$ given by 
  \[ U(\Gamma) = \{ \Gamma.A \} \qquad \tilde{U}(\Gamma) = \{ \textup{sections}\ s : \Gamma \to \Gamma.A \}\]
  with the evident projection map and any choice of pullbacks. Then $[\cC^\op,\Set]_U$ is isomorphic, as a contextual category, to $\cC$.
\end{prop}

When the original category $\cE$ is locally cartesian closed, it is possible to specify additional structure on the universe $U$ that would equip  the contextual category $\cE_U$ with the corresponding logical structure. The idea is that this introduces the desired structure in the universal case.  For instance, to describe an identity structure on a universe, we require the following notion.

\begin{dig}
Recall that for maps $i \colon A \to B$ and $f \colon X \to Y$ in $\cE$, a \textbf{lifting operation} witnessing a lifting property as below-left is a section to the map below-right:
\[ \begin{tikzcd} A \arrow[r, "x"] \arrow[d, "i"'] & X \arrow[d, "f"] \\ B \arrow[r, "y"'] \arrow[ur, dashed] & Y \end{tikzcd} \qquad \begin{tikzcd} \cE(B,X) \arrow[r, "{(-\cdot i, f\cdot -)}"] & \cE(A,X) \times_{\cE(A,Y)} \cE(B,Y) \end{tikzcd}\]
If $\cE$ is cartesian closed, an \textbf{internal lifting operation} is a section to the map
\[ \begin{tikzcd} X^B \arrow[r, "{(-\cdot i,f \cdot -)}"] & X^A \times_{Y^A} Y^B \end{tikzcd}\]
\end{dig}

Now suppose $\cE$ is a locally cartesian closed category with a universe $U$. If $U$ has an \emph{identity structure}, introduced below, we will see that the context category $\cE_{U}$ has an identity type structure in the sense of Definition \ref{defn:id-structure}. To motivate the definition, note that since all display maps are pullbacks of $\pi$ it suffices to consider orthogonality against $\pi$ in the slice over $U$. 

\begin{defn}\label{defn:id-structure-universe} An \textbf{identity structure} on a universe $\pi \colon \tilde{U} \to U$ consists of a map 
  \[ \Id \colon \tilde{U} \times_U \tilde{U} \to U\]
and a specified lift $r \colon \tilde{U} \to \Id^*\tilde{U}$ of the relative diagonal
\[
\begin{tikzcd} & \Id^*\tilde{U} \arrow[r] \arrow[d] \arrow[dr, phantom, "\lrcorner" very near start] & \tilde{U} \arrow[d, "\pi"] \\ \tilde{U} \arrow[ur, dashed, "r"] \arrow[r, "{(1,1)}"'] & \tilde{U} \times_U \tilde{U} \arrow[r, "\Id"'] & U
\end{tikzcd}
\]
 together with an internal lifting operation $J$ for $r$ against $\pi \times U$ in $\cE_{/U}$.
\end{defn}

\begin{thm}[{\cite[\S2.4]{voevodsky-ID}}]\label{thm:id-structure} An identity structure on a universe $U$ in $\cE$ induces an identity type structure on the contextual category $\cE_U$.
\end{thm}
\begin{proof} An identity structure on $U$ defines a factorization of the fibered diagonal $\pi \colon \tilde{U} \to U$, whose right factor is a display map. This pulls back, using the specified pullback squares of Definition \ref{defn:universe},  to define the factorization required by an identity type structure on $\cE_U$:
\[
\begin{tikzcd}[sep=small]  {(\Gamma;A)} \arrow[dr, "\refl_A"] \arrow[ddddrr, phantom, "\lrcorner" pos=.05] \arrow[ddddr, "p_A"'] \arrow[rrr, dotted] & & &[-35pt] \tilde{U} \arrow[dr, "r"] \arrow[ddddr, "\pi"' pos=.35]\\ & {(\Gamma; A,A,\Id_A)} \arrow[dr, "p_{\Id_A}"] \arrow[ddd] \arrow[rrr,dotted] \arrow[dddr, phantom, "\lrcorner" very near start] & & & \Id^*\tilde{U} \arrow[drrr, phantom, "\lrcorner" very near start]\arrow[rrr]\arrow[dr] \arrow[ddd]& & & \tilde{U} \arrow[dr, "\pi"]\\ & &  {(\Gamma;A,A)}  \arrow[ddl] \arrow[rrr, dotted] \arrow[ddr, phantom, "\lrcorner" very near start]& & & \tilde{U} \times_U \tilde{U} \arrow[ddl] \arrow[rrr, "\Id"] & & ~ & [+1em] U \\ ~ \\ & \Gamma \arrow[rrr, dotted, "A"'] &~ & ~& U
\end{tikzcd}    
\]

Internal homs in a locally cartesian closed category are stable under pullback: given $f \colon \Delta \to \Gamma$ and $A,X \in \cE_{/\Gamma}$, \[ f^*\map_\Gamma(A,X) \cong \map_\Delta(f^*A,f^*X).\] Consequently, an internal lifting operation in a slice pulls back to define an internal lifting operation between the pullbacks of the original maps. For any morphism $A \colon \Gamma \to U$ in $\cE$, the map $\refl_A \colon (\Gamma;A) \to (\Gamma;A,A,\Id_A)$ is defined to be the pullback of the map $r \in \cE_{/U}$ along $A \colon \Gamma \to U$, while the map $\pi \times U \in \cE_{/U}$ pulls back to the map $\pi \times \Gamma \in \cE_{/\Gamma}$, which is the ``generic display map'' in the slice over $\Gamma$. The internal lifting operation between $r$ and $\pi \times U$ in $\cE_{/U}$ then defines an internal lifting operation between $\refl_A$ and $\pi \times \Gamma$ in $\cE_{/\Gamma}$, which in particular specifies a solution to any lifting problem in $\cE$ of the form below-left and hence also below-center and below-right:
\[
  \begin{tikzcd} {(\Gamma;A)} \arrow[d, "\refl_A"'] \arrow[r]& [+1em] \tilde{U} \times \Gamma \arrow[d, "\pi \times \Gamma"] & [-2em] {(\Gamma;A)} \arrow[d, "\refl_A"'] \arrow[r]& [-1em]\tilde{U} \times \Gamma \arrow[d, "\pi \times \Gamma"'] \arrow[dr, phantom, "\lrcorner" very near start]  \arrow[r] & \tilde{U} \arrow[d, "\pi"] & [-2em] {(\Gamma;A)} \arrow[d, "\refl_A"'] \arrow[r]& [-1em] {(\Gamma;A,A,\Id_A,P)} \arrow[dr, phantom, "\lrcorner" very near start]\arrow[d, "p_P"] \arrow[r] & [-1em]  \tilde{U}\arrow[d, "\pi"] \\ {(\Gamma;A,A,\Id_A)} \arrow[r, "{(P, p_A\cdot p_A \cdot p_{\Id})}"'] \arrow[ur, dashed] & U \times \Gamma & {(\Gamma;A,A,\Id_A)} \arrow[r] \arrow[rr, bend right=10, "P"'] \arrow[urr, dotted, bend right=5] & U \times \Gamma \arrow[r] & U  & {(\Gamma;A,A,\Id_A)} \arrow[r, equals] \arrow[ur, dashed] & {(\Gamma;A,A,\Id_A)} \arrow[r, "P"'] & U
  \end{tikzcd}
\]
When interpreted in the contextual category $\cE_{U}$, this right-hand lift provides specified solutions to the lifting problems in Definition  \ref{defn:id-structure}. Since the lift is defined as the pullback of a universal specified lift to a map between objects defined as strict pullbacks in the contextual category $\cE_{U}$, this data is strictly stable under substitution as required.
\end{proof}

Of course, it remains to construct a suitable universe in our locally cartesian closed category. In part II, we explain how this is done in the category of simplicial sets, leading to the simplicial model of univalent foundations.

\part*{Lecture II: The simplicial model of univalent foundations}

\section{A universal fibration}\label{sec:universe}

Recall from \S\ref{ssec:cc-universe} that Voevodsky defined a \textbf{universe} in a category $\cE$ to be a morphism $\pi \colon \tilde{U} \to U$ together with a chosen pullback along any $A \colon \Gamma \to U$. In the resulting contextual category, the map $A \colon \Gamma \to U$ encodes an extension of the context $\Gamma$ by $A$ with the chosen pullback of $\pi$ defining the corresponding display map.

When $\cE$ is the category of simplicial sets, the intuitions developed in \S\ref{ssec:why-simplicial-sets} suggest that we would like the display maps to be Kan fibrations, which we refer to simply as ``fibrations'' whenever there is no competing notion in sight. Thus, we would like the universe to be a simplicial set $U$ that classifies fibrations. By the Yoneda lemma, this suggests that $U_n$ should be defined to be the set of fibrations over $\Delta^n$, but there are two problems with this na\"{i}ve construction:
\begin{enumerate}
  \item Size: We must restrict to ``small'' fibrations---whose fibers are $\kappa$-presentable for some regular cardinal $\kappa$---and then ensure that we are taking large sets of all such, rather than proper classes, in order for $U_n$ to be a set.
  \item Coherence: The action of the simplicial operators $\alpha \colon \Delta^m \to \Delta^n$ by pullback is only functorial up to isomorphism, not on the nose, so the na\"{i}ve construction does not guarantee a strict functor.
\end{enumerate}
We largely sweep the size issues under the rug---referring to a ``\emph{set} of \emph{small} fibrations'' without delving too deeply into its construction\footnote{The original sources \cite{klv} and \cite{shulman} provide a precise accounting.}---to reserve our attention for the coherence issues, which are quite interesting. 

\subsection{What do fibrations form?}

The coherence issues derive from the fact that chosen pullbacks are not typically strictly functorial but rather \emph{pseudofunctorial}. Put more precisely, in a category $\cE$ with pullbacks, there is a (large) groupoid-valued pseudofunctor
\[ \begin{tikzcd}[row sep=tiny]  \cE^\op \arrow[r, "\EE"] & \GPD  \\ B \arrow[r, maps to] & (\cE_{/B})^\cong\arrow[d, "f^*"] \\ [7pt] A \arrow[u, "f"] \arrow[r, maps to] & (\cE_{/A})^\cong \end{tikzcd}\] that sends an object $B$ to the groupoid $(\cE_{/B})^\cong$ of morphisms with codomain $B$ and acts on arrows by pullback. The pseudofunctor $\EE$ is called the \textbf{core of self-indexing} of $\cE$, the (contravariant) ``self-indexing'' being the corresponding pseudofunctor $B \mapsto \cE_{/B} : \cE^\op \to \CAT$. When $\cE$ has a pullback stable class of fibrations $\sF$, they similarly assemble into a pseudofunctor
\[ \begin{tikzcd}[row sep=tiny]  \cE^\op \arrow[r, "\FF"] & \GPD  \\ B \arrow[r, maps to] & (\sF_{/B})^\cong\arrow[d, "f^*"] \\ [7pt] A \arrow[u, "f"] \arrow[r, maps to] & (\sF_{/A})^\cong \end{tikzcd}\] 
defined by restricting to the full subgroupoid $(\sF_{/B})^\cong \subset (\cE_{/B})^\cong$  of fibrations over $B$. For any regular cardinal $\kappa$, there are analogous pseudofunctors $\FF^\kappa \hookrightarrow \EE^\kappa$ defined by restricting to small fibrations or small maps as appropriate. 

By the bicategorical Yoneda lemma, a small fibration $p \colon E \fto B$ defines a pseudonatural transformation \[ \begin{tikzcd} \cE(-,B) \arrow[r, "p"] & \FF^\kappa .\end{tikzcd}\] We might imagine that the \emph{universal fibration} $\pi \colon \tilde{U} \fto U$ is a representing element, defining a pseudonatural isomorphism $\cE(B,U) \cong \FF^\kappa(B)$. But this can't be, at least not if $\cE$ is very interesting:
\begin{itemize}
  \item Since $\cE$ is a 1-category $\cE(B,U)$ is a set, while $\FF^\kappa(B) \coloneqq (\sF_{/B})^{\cong,\kappa}$ is the groupoid of small fibrations over $B$. 
  \item Even ignoring this, we can see that the same fibration might be classified by multiple maps to $U$ and the same map might classify multiple fibrations:
\[ 
  \begin{tikzcd} E \arrow[r, "\cong"] \arrow[dr, phantom, "\lrcorner" very near start]\arrow[d, two heads, "p"'] & E \arrow[d, two heads, "p"'] \arrow[r] \arrow[dr, phantom, "\lrcorner" very near start] & \tilde{U} \arrow[d, "\pi", two heads] & & E' \arrow[dr, phantom, "\lrcorner" very near start] \arrow[r, "\cong"] \arrow[d, two heads, "p'"'] & E \arrow[d, two heads, "p"'] \arrow[r] \arrow[dr, phantom, "\lrcorner" very near start] & \tilde{U} \arrow[d, "\pi", two heads]  \\ B \arrow[r, "\cong"] & B \arrow[r, "{\name{p}}"'] & U & &  B \arrow[r, equals] & B \arrow[r, "{\name{p}}"'] & U
  \end{tikzcd}
\]
\end{itemize}

In a contextual category built from a universe, the display maps arose as specified pullbacks of the universe: a type in context $\Gamma$ was encoded by the data of a morphism $A \colon \Gamma \to U$. The corresponding display map was then the pullback axiomatized by the universe structure
\[
  \begin{tikzcd}
    (\Gamma;A) \arrow[r, "q_A"] \arrow[d, "p_A"'] \arrow[dr, phantom, "\lrcorner" very near start] & \tilde{U} \arrow[d, "\pi"] \\ \Gamma \arrow[r, "A"'] & U
  \end{tikzcd}
\]
This guarantees that each display map may be realized as some pullback of $\pi \colon \tilde{U} \to U$ but does not imply that such pullbacks are unique. Thus, we ask the universe to \emph{weakly} classify the small fibrations, demanding that the pseudonatural transformation $\pi \colon \cE(-,U) \to \FF^\kappa$ associated to $\pi \colon \tilde{U} \to U$ is surjective. Actually, we demand a bit more.

\begin{defn}\label{defn:classifying-fibration} In a model category $\cE$ with all objects cofibrant, a small fibration $\pi \colon \tilde{U} \to U$ is a \textbf{universal small fibration} if the pseudonatural transformation \[\pi \colon \cE(-,U) \to \FF^\kappa\] is a trivial fibration, meaning that for all cofibrations $i \colon A \cto B$ in $\cE$, any strictly commutative square of pseudonatural transformations admits a lift:
  \[\begin{tikzcd}
 \cE(-,A) \arrow[d, "i\circ -"'] \arrow[r, "h \circ -"] & \cE(-,U) \arrow[d, "\pi"] \\ \cE(-,B) \arrow[ur, dashed, "k"'] \arrow[r, "p"'] & \FF^\kappa
  \end{tikzcd}
\]
  \end{defn}

The condition amounts to the following property: 
given any pair of pullback squares between small fibrations as below
\begin{equation}\label{eq:realignment-property}
\begin{tikzcd}
  D \arrow[dd, two heads, "q"'] \arrow[rr] \arrow[drr, phantom, "\lrcorner" very near start] \arrow[dddr, phantom, "\lrcorner" very near start]\arrow[dr] & & \tilde{U} \arrow[dd, two heads, "\pi"] \\ & E  \arrow[dr, phantom, "\lrcorner" very near start]\arrow[ur, dashed] & ~ \\ A \arrow[dr, tail, "i"'] \arrow[rr, "h" near end] & & U \\ & B \arrow[ur, dashed, "k"'] \arrow[from=uu, two heads, crossing over, "p"' near start]
\end{tikzcd}  
\end{equation}
there exists an extension $k$ of $h$ along $i$ factoring the back pullback square as a composite of pullbacks. This has been referred to as \textbf{strong gluing} by Sterling and Angiuli \cite{angiuli-sterling}, the \textbf{stratification property} by \cite[2.3.3]{stenzel}, and \textbf{acyclicity} by Shulman \cite{shulman-reedy,shulman}. A more recent name for this conditition is \textbf{realignment} \cite{GSS}.

\begin{rmk}\label{rmk:weak-classification}
  As our terminology suggests, a universal fibration weakly classifies fibrations, at least when all objects are cofibrant. Taking $A=\emptyset$, \eqref{eq:realignment-property} implies that any $p \colon E \fto B$ is classified by a pullback square into $\pi \colon \tilde{U} \fto U$.
\end{rmk}

We'll explain the further relevance of realignment in \S\ref{ssec:fibrancy} and \S\ref{ssec:univalence}, but first note that we have entirely sidestepped the question of how a fibration classifier might be constructed. It is to that subject that we now turn.

\subsection{Hofmann-Streicher universes}

There is a general method, due to Hofmann and Streicher \cite{HS-lifting}, that defines a weak classifier for small maps in any presheaf category, which satisfies a realignment property analogous to Definition \ref{defn:classifying-fibration}. We describe this construction in a particularly slick way observed by Awodey \cite{awodey-hofmann} and independently by Sattler. Rather than introduce general notation for presheaf categories, we work in the category $\cE = \hat{\DDelta}$ of simplicial sets.

Let $\kappa$ be an infinite regular cardinal,\footnote{As the title of their article suggests, Hofmann and Streicher prefer to work with an inaccessible cardinal \cite{HS-lifting}, but for the purposes of obtaining a $\kappa$-small map classifier it suffices to use a regular cardinal $\kappa$ larger than the cardinality of the morphisms in the indexing category.} and write $\set \subset \Set$ for a full subcategory containing at least one set of each cardinality $\lambda < \kappa$ and at most $\kappa$ many sets of each cardinality. Write $\sset \subset \sSet$ for the category of $\kappa$-small simplicial sets. By regularity, a map $f \colon A \to B$ of simplicial sets lies in $\sset$, i.e., is a $\kappa$-\textbf{small map} or simply a \textbf{small map}, just when its base and fibers are $\kappa$-small. 
 
\begin{defn} Consider the (covariant) self-indexing functor $\DDelta_{/\bullet} \colon \DDelta \to \Cat$ sending $[n] \in \DDelta$ to the slice category $\DDelta_{/[n]}$ and the simplicial operators $\alpha \colon [n] \to [m]$ to the composition functors $\Sigma_\alpha \colon \DDelta_{/[n]} \to \DDelta_{/[m]}$. These slice categories are the \emph{categories of elements} of the simplicial sets $\Delta^n$ and the left Kan extension of this functor along the Yoneda embedding defines the category of elements functor and its right adjoint: 
  \[
    \begin{tikzcd} & \DDelta \arrow[dr, "{\DDelta_{/\bullet}}"] \arrow[dl, hook', "\yo"'] \\ \sSet \arrow[rr, bend left=10, "{\el{}\coloneqq\colim_{\DDelta_{/\bullet}}}"] \arrow[rr, phantom, "\bot"] && \Cat \arrow[ll, bend left=10, "{N_{\el{}} \coloneqq \hom(\DDelta_{/\bullet}, -)}"]
  \end{tikzcd}
 \]
The $\kappa$-small map classifier is then defined by applying the right adjoint to the opposite of the forgetful functor from pointed sets to sets:
\[\begin{tikzcd}[column sep=small] \tilde{V} \arrow[r, phantom, "\coloneqq"] \arrow[d, "\varpi"'] & N_{\el{}}\set_{\ast}^\op \arrow[d] \\ V\arrow[r, phantom, "\coloneqq"] & N_{\el{}}\set_{\ast}
\end{tikzcd}
\]
\end{defn}

Explicitly, an $n$-simplex in $V$ is a presheaf $A \colon \DDelta_{/[n]}^\op \to \set$. By the coherence problem noted above, it is insufficient to define $V_n$ to be the set of $\kappa$-small maps over $\Delta^n$ because the action of a simplicial operators $\alpha \colon \Delta^m \to \Delta^n$, by pullback, is only pseudofunctorial. However, under the equivalence of categories\footnote{More generally, the slice category over any presheaf $X$ is equivalent to the category of presheaves on the category of elements of $X$. In particular, ``$\DDelta_{/[n]}^\op$'' should be read as an abbreviation for $(\DDelta_{/[n]})^\op$, the slice category $\DDelta_{/[n]}$ being the category of elements of the simplicial set $\Delta^n$.}
\[
  \begin{tikzcd}[column sep=tiny] \sset_{/\Delta^n} \arrow[r, phantom, "\simeq"] \arrow[d, "\alpha^*"'] & \set^{\DDelta_{/[n]}^\op} \arrow[d, "-\cdot \alpha_!"] \\ 
    \sset_{/\Delta^m} \arrow[r, phantom, "\simeq"] & \set^{\DDelta_{/[m]}^\op}  
  \end{tikzcd}
\]
the pullback action is replaced by precomposition with the composition functor $\alpha_! \colon \DDelta_{/[m]} \to \DDelta_{/[n]}$, which is strictly functorial, defining a large simplicial set $V \in \sSet$. An element $A \colon \Delta^n \to V$ defines a small map over $\Delta^n$ whose fiber over $\alpha \in (\Delta^n)_m$ is the set $A_\alpha$  defined by the functor $A \colon \DDelta_{/[n]}^\op \to \set$. 

An $n$-simplex in $\tilde{V}$ is a presheaf $A \colon \DDelta_{/[n]}^\op \to \set_\ast$ valued in pointed sets. Since the indexing category $\DDelta_{/[n]}$ has a terminal object $\id_{[n]}$, an $n$-simplex in $\tilde{V}$ is equally given by the data of $A \colon \Delta^n \to V$ together with a  section 
\[
  \begin{tikzcd}
  & \tilde{V} \arrow[d, "\varpi"]\\ \Delta^n \arrow[ur, dashed,  "a"]  \arrow[r, "A"'] & V
 \end{tikzcd}
\]


Note that even though $V$ is large, the evident forgetful map $\pi \colon \tilde{V} \to V$ is $\kappa$-small since this means that the pullback to any $\kappa$-small object is $\kappa$-small.  We prove that $\pi \colon \tilde{V} \to V$ weakly classifies small maps in $\cE = \hat{\DDelta}$ with arbitrary codomain by proving that it has the realignment property \eqref{eq:realignment-property} with respect to monomorphisms $i \colon A \cto B$ and small maps over $A$ and $B$.

\begin{lem}[{\cite[3.9]{cisinski},\cite[8.4]{orton-pitts}, \cite[6]{awodey-hofmann}}]\label{lem:Hofmann-Streicher-realignment}
 The Hofmann-Streicher universe $\pi \colon \tilde{V} \to V$ satisfies realignment relative to the the $\kappa$-small maps $\EE^\kappa \subset \EE$ in the core of self-indexing: 
 \[\begin{tikzcd}
  \cE(-,A) \arrow[d, "i\circ -"'] \arrow[r, "h \circ -"] & \cE(-,V) \arrow[d, "\pi"] \\ \cE(-,B) \arrow[ur, dashed, "k"'] \arrow[r, "p"'] & \EE^\kappa
   \end{tikzcd}
 \]
That is, given  any pair of pullback squares between maps with $\kappa$-small fibers as below
 \[
 \begin{tikzcd}
   D \arrow[dd, "q"'] \arrow[rr, "g"] \arrow[drr, phantom, "\lrcorner" very near start] \arrow[dddr, phantom, "\lrcorner" very near start]\arrow[dr, tail, "j"'] & & \tilde{V} \arrow[dd,  "\pi"] \\ & E  \arrow[dr, phantom, "\lrcorner" very near start]\arrow[ur, dashed, "\ell"'] & ~ \\ A \arrow[dr, tail, "i"'] \arrow[rr, "h" near end] & & V \\ & B \arrow[ur, dashed, "k"'] \arrow[from=uu, crossing over, "p"' near start]
 \end{tikzcd}  
 \]
 where $i \colon A \cto B$ is a monomorphism, there exists an extension $k$ of $h$ along $i$ factoring the back pullback square as a composite of pullbacks. 
\end{lem}

\begin{proof}
The realignment problem in $\sSet$ transposes via the adjunction $\el{} \vdash N_{\el{}}$ to a realignment problem in $\Cat$
  \[
    \begin{tikzcd}
      \int D \arrow[dd, "\int q"'] \arrow[rr, "g"] \arrow[drr, phantom, "\lrcorner" very near start] \arrow[dddr, phantom, "\lrcorner" very near start]\arrow[dr, tail, "\int j"' near end] & & \set_\ast^\op \arrow[dd,  "\pi"] \\ & \int E  \arrow[dr, phantom, "\lrcorner" very near start]\arrow[ur, dashed, "\ell"'] & ~ \\ \int A \arrow[dr, tail, "\int i"'] \arrow[rr, "h" near end] & & \set^\op \\ & \int B \arrow[ur, dashed, "k"'] \arrow[from=uu, crossing over, "\int p"' near start]
    \end{tikzcd}  
    \]
By inspection the category of elements functor preserves pullback and thus carries $i$ and $j$ to monomorphisms $\int i \colon \int A \rightarrowtail \int B$ and $\int j\colon \int D \rightarrowtail \int E$.

Under the equivalence of categories noted above $p \colon E \to B$ defines an object $E \in \set^{(\int B)^\op}$ in the category of small presheaves on the category of elements of $B$. Applying the category of elements construction, with $\int B$ in place of $\DDelta$, we obtain a $\kappa$-small discrete fibration, which is then classified by a pullback square 
\[
  \begin{tikzcd} \int E \arrow[d, "\int p"'] \arrow[r, "e"] \arrow[dr, phantom, "\lrcorner" very near start] & \set_\ast^\op \arrow[d, "\pi"] \\ \int B \arrow[r, "b"'] & \set^\op
  \end{tikzcd}
\]
One can check that the category of elements of the $\int B$-indexed presheaf $E$ is isomorphic to the category of elements of the $\DDelta$-indexed presheaf $E$. By uniqueness of pullbacks, $\int D$ is isomorphic to the pullback of $\pi$ along $b \cdot \int i$. Under the equivalence between the category of discrete fibrations over $\int{B}$ and the category of small presheaves on $\int{B}$, this isomorphism provides the data of a natural isomorphism $\Phi \colon b \cdot \int{i} \cong h$. Since $\pi$ is a discrete fibration, $\Phi$ lifts to define a unique natural isomorphism $\Psi \colon e \cdot \int{j} \cong g$ over $\Phi$.
 
These isomorphisms define a lifting problem in the category $\Cat^\2$ in which the left-hand map is a pullback square between functors $\int j$ and $\int i$ that are injective on objects, while the right-hand map is a pullback square between the surjective equivalences $\ev_0$ defined by endpoint evaluation from the category of isomorphisms.
\[
    \begin{tikzcd}
      \int D \arrow[dd, "\int q"'] \arrow[rr, "g"] \arrow[drr, phantom, "\lrcorner" very near start] \arrow[dddr, phantom, "\lrcorner" very near start]\arrow[dr, tail, "\int j"' near end] & & (\set_{\ast}^\iso)^\op \arrow[drr, phantom, "\lrcorner" pos=.001] \arrow[ddr, phantom, "\lrcorner" pos=.01] \arrow[dd,  "\pi"  near end]\arrow[dr, utfibarrow, "\ev_0"] \arrow[rr, "\ev_1", utfibarrow] & [-10pt] ~ & \set_{\ast}^\op \arrow[dd, "\pi"] \\ & \int E \arrow[urrr, dotted, bend right=5, "\ell"' very near end] \arrow[rr, "e"' near start, crossing over] \arrow[dr, phantom, "\lrcorner" very near start]\arrow[ur, dashed] & ~ & \set_{\ast}^\op  & ~ \\ \int A \arrow[dr, tail, "\int i"'] \arrow[rr, "h" near end] & & (\set^\iso)^\op \arrow[dr, tfibarrow, "\ev_0"'] \arrow[rr, utfibarrow, "\ev_1" near end, "\sim"' near end] & ~ & \set^\op \\ & \int B \arrow[rr, "b"'] \arrow[ur, dashed] \arrow[urrr, dotted, bend right=5, "k"' very near end] \arrow[from=uu, crossing over, "\int p"' near start] & & \set^\op \arrow[from=uu, "\pi" near start, crossing over]
    \end{tikzcd}  
    \]
These maps belong to the left and right classes of a weak factorization system on $\Cat$, deriving from the (injective-on-objects, surjective equivalence) algebraic weak factorization system on $\Cat$. Thus, there exists a solution as indicated, which defines another pullback square. This lift defines the required pair of functors $k$ and $\ell$ so that $k \cdot \int i = h$ and $\ell \cdot \int j = g$, which transpose to define the desired map of simplicial sets.
\end{proof}

\section{The simplicial model of univalent foundations}\label{sec:simplicial-model}

The technique of Hofmann-Streicher universes can be used to define a universal Kan fibration, which is the key ingredient in the simplicial model of univalent foundations. This is not the approach taken to defining the universe in \cite{klv} but was quickly noted as alternative possible route; see \cite{cisinski} or \cite{streicher}.

\subsection{A universal Kan fibration}

The Hofmann-Streicher universe may be restricted to define the universal Kan fibration by taking $U \subset V$ to be the simplicial subset spanned by those small maps over simplices that are Kan fibrations and by defining $\pi \colon \tilde{U} \to U$ to be the pullback
\begin{equation}\label{eq:universal-Kan-fibration} \begin{tikzcd} \tilde{U} \arrow[d, "\pi"'] \arrow[r] \arrow[dr, phantom, "\lrcorner" very near start] & \tilde{V} \arrow[d, "\pi"] \\ U \arrow[r, hook] & V
\end{tikzcd}
\end{equation}

\begin{rmk}\label{rmk:kan-local}
Crucially, the Kan fibrations of simplicial sets are \textbf{local}: a map $p \colon E \fto B$ is a Kan fibration if and only if for all $n$ and all $b \colon \Delta^n \to B$, the pullback defines a Kan fibration. This follows from the fact that the fibrations are characterized by a right lifting property against maps with representable codomains:
\[
  \begin{tikzcd} \Lambda^n_k \arrow[d, utcofarrow] \arrow[r] & \bullet \arrow[r] \arrow[d, two heads]\arrow[dr, phantom, "\lrcorner" very near start] & E \arrow[d, two heads, "p"] \\ \Delta^n \arrow[r, equals] \arrow[ur, dashed] & \Delta^n \arrow[r, "b"'] & B
  \end{tikzcd}
\]
\end{rmk}

Consequently:

\begin{cor} $\pi \colon \tilde{U} \fto U$ is a Kan fibration. 
\end{cor}
\begin{proof}
  By construction, each pullback
  \[\begin{tikzcd} E \arrow[d, "p"'] \arrow[r] \arrow[dr, phantom, "\lrcorner" very near start] & \tilde{U} \arrow[d, "\pi", two heads]\\ \Delta^n \arrow[r, "{\name{p}}"'] & U
  \end{tikzcd}\]
  is a Kan fibration.
\end{proof}

Another reflection of this locality is the following:

\begin{lem}\label{lem:locality}
  Let $p \colon E \to B$ be a small map and consider any classifying square
  \[\begin{tikzcd}
    E \arrow[r] \arrow[d, "p"', two heads] \arrow[dr, phantom, "\lrcorner" very near start] & \tilde{V} \arrow[d, "\pi"] \\ B \arrow[r, "{\name{p}}"'] & V
  \end{tikzcd}
  \]
Then $p$ is a Kan fibration if and only if the classifying square factors through \eqref{eq:universal-Kan-fibration}.
\end{lem}
\begin{proof}
  If $p$ is a pullback of $\pi \colon \tilde{U} \fto U$ it is clearly a Kan fibration. Conversely, if $p$ is a Kan fibration then so is its restriction along any $b \colon \Delta^n \to B$. Recall $U$ is defined as a simplicial subset of $V$, so $\name{p} \colon B \to V$ factors through $U \hookrightarrow V$ just when each for each $b \colon \Delta^n \to B$ the corresponding map $\name{p} \cdot b \colon \Delta^n \to V$ so factors. But since a Kan fibration pulls back to Kan fibrations, this means the corresponding elements are in the simplicial subset.
\end{proof}

Put more abstractly, locality provides a pullback square 
\[ 
  \begin{tikzcd} \cE(-,U) \arrow[d, "\pi"'] \arrow[dr, phantom, "\lrcorner" very near start]\arrow[r, hook] & \cE(-,V) \arrow[d, "\pi"] \\ \FF^\kappa \arrow[r, hook] & \EE^\kappa
\end{tikzcd} 
\]
in the category of groupoid-valued pseudofunctors and pseudonatural transformations between them. Thus the realignment condition of Definition \ref{defn:classifying-fibration} follows from the corresponding result for Hofmann-Streicher universes

\begin{cor} The Kan fibration $\pi \colon \tilde{U} \to U$ is a universal $\kappa$-small Kan fibration.
\end{cor}
\begin{proof}
For any lifting problem as presented by the left-hand square, the dashed lift exists by Lemma \ref{lem:Hofmann-Streicher-realignment}. 
\[ 
  \begin{tikzcd}\cE(-,A) \arrow[r] \arrow[d, "i\circ -"'] & \cE(-,U) \arrow[d, "\pi"'] \arrow[dr, phantom, "\lrcorner" very near start]\arrow[r, hook] & \cE(-,V) \arrow[d, "\pi"] \\  \cE(-,B) \arrow[r] \arrow[urr, dashed, bend right=5] \arrow[ur, dotted] & \FF^\kappa \arrow[r, hook] & \EE^\kappa
\end{tikzcd} 
\]
On account of the pullback established by Lemma \ref{lem:locality}, this induces the required dotted lift.
\end{proof}

Since the cofibrations in the category of simplicial sets are the monomorphisms, all simplicial sets are cofibrant. Thus taking $A=\emptyset$ we see by Remark \ref{rmk:weak-classification} that:

\begin{cor}\label{cor:weak-classifying} Any Kan fibration $p \colon E \fto B$ with small fibers is classified by a pullback square
  \[
    \begin{tikzcd} E \arrow[r] \arrow[d, two heads, "p"'] \arrow[dr, phantom, "\lrcorner" very near start] & \tilde{U} \arrow[d, two heads, "\pi"] \\ B \arrow[r, "\name{p}"'] & U
    \end{tikzcd} \qed\]
  \end{cor}

\subsection{Fibrancy of the universe}\label{ssec:fibrancy}

In the simplicial model of univalent foundations, types are meant to correspond to Kan complexes. One use of realignment is in proving that $U$ is a Kan complex. Recall that Kan complexes of simplicial sets are characterized by right lifting against the \textbf{acyclic cofibrations}, monomorphisms that also define weak homotopy equivalences
\[ 
  \begin{tikzcd} A \arrow[r, "h"]\arrow[d, tcofarrow, "j"'] & U \\ B \arrow[ur, dashed]
  \end{tikzcd}
\] 

If maps into $U$ corresponded bijectively to small fibrations over the domain, this would unravel to the following: 

\begin{defn}\label{defn:FEP} A model category $\cE$ satisfies the \textbf{fibration extension property} if for all small fibrations $q \colon D \fto A$ and all trivial cofibrations $j \colon A \cwto B$, there exists a small fibration $p \colon E \fto B$ that pulls back along $j$ to $q$.
  \[
    \begin{tikzcd} D \arrow[d, two heads, "q"'] \arrow[r, dashed] \arrow[dr, phantom, "\lrcorner" very near start] & E \arrow[d, dashed, two heads, "p"] \\ A \arrow[r, tcofarrow, "j"'] & B
    \end{tikzcd}
  \]
\end{defn}

Note that Definition \ref{defn:FEP} makes no reference to a particular universe $U$. Nevertheless, when $\pi \colon \tilde{U} \fto U$ is a universal fibration in the sense of Definition \ref{defn:classifying-fibration}, satisfying realignment \eqref{eq:realignment-property}, the fibration extension property encodes the fibrancy of $U$. 

\begin{lem}\label{lem:FEP-fibrant} Suppose $\pi \colon \tilde{U} \to U$ is a universal fibration in a model category $\cE$ with all objects cofibrant. Then $U$ is fibrant if and only if $\cE$ satisfies the fibration extension property.
\end{lem}
\begin{proof} Suppose $\cE$ satisfies the fibration extension property and consider a lifting problem:
  \[ 
    \begin{tikzcd}
      A \arrow[r, "h"] \arrow[d, "j"', tcofarrow] & U \\ B \arrow[ur, dashed, "k"']
    \end{tikzcd}
  \]
Define a small fibration $q \colon D \fto A$ by pulling back $\pi$ along $h$. Then use the fibration extension property to extend this to a small fibration $p \colon E \fto B$ that pulls back along $j$ to $q$. By realignment, it follows that the classifying map $h$ extends along $j$ to a classifying map $k$ for $p$ so that $k \cdot j = h$, solving the lifting problem. This proves that the fibration extension property implies fibrancy of the universe.

Conversely, if $U$ is fibrant and we are given a small fibration $q \colon D \fto A$, we may use Remark \ref{rmk:weak-classification} to choose a classifying pullback square. In particular, this choice defines a classifying map and thus a lifting problem
\[ 
    \begin{tikzcd}
      A \arrow[r, "{\name{q}}"] \arrow[d, "j"', tcofarrow] & U \\ B \arrow[ur, dashed, "\name{p}"']
    \end{tikzcd}
  \]
  which admits a solution. The pullback of $\pi$ along this map, displayed below-right, defines a small fibration over $B$. The pullback square for $q$ factors through the one for $p$ defining the desired extension square:
  \[
    \begin{tikzcd} D \arrow[rr, bend left]\arrow[d, two heads, "q"'] \arrow[r, dashed] \arrow[dr, phantom, "\lrcorner" very near start] & E \arrow[d, dashed, two heads, "p"'] \arrow[r] \arrow[dr, phantom, "\lrcorner" very near start] & \tilde{U} \arrow[d, two heads, "\pi"] \\ A \arrow[r, utcofarrow, "j"] \arrow[rr, bend right, "\name{q}"'] & B \arrow[r, "\name{p}"] & U
    \end{tikzcd} \qedhere
  \]
\end{proof}

Thus, to show that $U$ is a Kan complex, we need only verify the fibration extension property in the category of simplicial sets. Since the Kan complexes are detected by right lifting against the simplicial horn inclusions $\Lambda^n_k \cwto \Delta^n$, it suffices to consider extensions of fibrations along such maps, as the proof just given reveals.

\begin{thm} The category of simplicial sets satisfies the fibration extension property.
\end{thm}
\begin{proof}
By Lemma \ref{lem:FEP-fibrant}, it suffices to extend a Kan fibration over a horn $\Lambda^n_k$ to a Kan fibration over the simplex $\Delta^n$ in such a way that the latter pulls back to the former. To do so, we use a result of Quillen \cite{quillen-minimal} to factor $q \colon E \fto \Lambda^n_k$ as a trivial fibration followed by a minimal fibration. Since the base of the minimal fibration is contractible, it is isomorphic to a projection map $F \times \Lambda^n_k \to \Lambda^n_k$ for some Kan complex $F$ and thus extends as displayed:
\[
  \begin{tikzcd} E \arrow[d, tfibarrow] \arrow[r, tcofdashed] \arrow[dr, phantom, "\lrcorner" very near start] & E' \arrow[d, tfibdashed] \\ F \times \Lambda^n_k \arrow[d, two heads] \arrow[r, tcofarrow] \arrow[dr, phantom, "\lrcorner" very near start] & F \times \Delta^n \arrow[d, two heads] \\ \Lambda^n_k \arrow[r, tcofarrow] & \Delta^n
  \end{tikzcd}\]
Thus it remains only to extend the trivial fibration, but this can be achieved defining $E' \fwto F \times \Delta^n$ to be the pushforward of $E \fwto F \times \Lambda^n_k$ along the monomorphism $F \times \Lambda^n_k \cto F \times \Delta^n$. Since monomorphisms are stable under pullback, the trivial fibrations are stable under pushforward. Finally, if a map is post-composed with a monomorphism and then pulled back along that monomorphism, the result is isomorphic to the map we started with; thus, taking right adjoints, the pushforward along a monomorphism pulls back along that monomorphism to a map isomorphic to the one we started with, which identifies $E$ as the pullback of $E'$.
\end{proof}

\subsection{Internal universes}

By Construction \ref{cons:universe-category}, a universe $U$ in $\cE$ gives rise to a contextual category $\cE_{/U}$. An internal universe structure on $U$ will provide a universe, in the sense of Definition \ref{defn:universe}, inside the category $\cE_{/U}$. 

\begin{defn} An \textbf{internal universe} in $U$ consists of arrows $u_0 \colon 1 \to U$ and $i \colon U_0 \to U$ where $U_0$ is defined by the canonical pullback
  \[
    \begin{tikzcd}
      U_0 \arrow[r] \arrow[d, "p_0"'] \arrow[dr, phantom, "\lrcorner" very near start] & \tilde{U} \arrow[d, "\pi"] \\ 1 \arrow[r, "u_0"'] & U
    \end{tikzcd}
  \]
\end{defn}

Note that the map $u_0$ defines an object $(1;u_0) \in \cE_{U}$ in the contextual category built from the universe $U$. At the same time, the canonical pullback along the map $i \colon U_0 \to U$ in $\cE$ defines a square
\[
  \begin{tikzcd} \tilde{U}_0 \arrow[dr, phantom, "\lrcorner" very near start] \arrow[r] \arrow[d, "\pi_0"'] & \tilde{U} \arrow[d, "\pi"] \\ U_0 \arrow[r, "i"'] & U
  \end{tikzcd}
\]
via which the universe $U$ induces a universe structure on $U_0$ as an object in $\cE$. We say that the internal universe $U_0$ is \textbf{closed under identity types} if it carries an identity structure, as in Definition \ref{defn:id-structure-universe}, commuting with $i$. Theorem \ref{thm:id-structure} proves that an identity  structure on $U$ induces an identity type structure on $\cE_U$. Moreover:

\begin{thm}[{\cite{voevodsky-ID}}] Suppose $\cE$ has a universe $U$ with an identity structure and an internal universe $U_0$ in $U$ closed under identity types. Then $\cE_U$ has an identity type structure as well as a universe \`{a} la Tarski that is closed under identity types.
\end{thm}

In the category of simplicial sets, we constructed a universal $\kappa$-small fibration $\pi_\kappa \colon \tilde{U}_\kappa \to U_\kappa$ for any regular cardinal $\kappa$. Now suppose $\kappa$ is an inaccessible cardinal and $\lambda < \kappa$ is also inaccessible. Since $\lambda < \kappa$, $U_\lambda$ is itself $\kappa$-small and so is representable as a pullback along some $u_\lambda \colon 1 \to U_\kappa$ in the category of simplicial sets
\[
  \begin{tikzcd} U_\lambda \arrow[d, "p_\lambda"'] \arrow[r] \arrow[dr, phantom, "\lrcorner" very near start] & \tilde{U}_\kappa \arrow[d, "\pi_\kappa"] \\ 1 \arrow[r, "u_\lambda"'] & U_\kappa
  \end{tikzcd}
\]
By construction of these universes, there is a natural inclusion $i \colon U_\lambda \to U_\kappa$. In summary:

\begin{thm}[{\cite[2.3.4]{klv}}] The universe $U_\kappa$ carries an identity type structure, while $U_\lambda$ give an internal universe that is closed under identity types.
\end{thm}

As discussed in \cite[2.3.4]{klv}, the universe $U_\kappa$ also bears structure corresponding to the other logical rules and $U_\lambda$ is closed under the corresponding logical structure.

\subsection{Univalence}\label{ssec:univalence}

We now explain the interpretation of the univalence axiom in a model category $\cE$,  for instance the model category of simplicial sets that presents the $\infty$-topos of spaces. Our exposition follows \cite{shulman-reedy}, which explains univalence as follows:

\begin{quote}
  The univalence axiom, when interpreted in a model category, is a statement about a ``universe object'' $U$, which is fibrant and comes equipped with a fibration $\pi \colon \tilde{U} \twoheadrightarrow U$ that is generic, in the sense that any fibration with ``small fibers'' is a pullback of $\pi$. \ldots In homotopy theory, it would be natural to ask for the stronger property that $U$  is a classifying space for small fibrations, i.e. that homotopy classes of maps $A \to U$
  are in bijection with (rather than merely surjecting onto) equivalence classes of
  small fibrations over $A$. The univalence axiom is a further strengthening of this: it
  says that the path space of $U$ is equivalent to the ``universal space of equivalences''
  between fibers of $\pi$ \ldots In particular, therefore, if two
  pullbacks of $\pi$ are equivalent, then their classifying maps are homotopic. \cite[84]{shulman-reedy}
\end{quote}

The univalence axiom concerns the object $\Eq(\tilde{U}) \fto U \times U$ of equivalences defined for the universal fibration $\pi \colon \tilde{U} \to U$. The idea is that a generalized element 
\[
  \begin{tikzcd}  & \Eq(\tilde{U}) \arrow[d, two heads] \\ B \arrow[ur, dashed, "\name{\epsilon}"] \arrow[r, "{(\name{p_1},\name{p_2})}"']& U \times U
  \end{tikzcd} \]
 should represent a fibered equivalence $\epsilon \colon E_1 \simeq E_2$ between the fibrations over $B$ classified by these maps.

We explain how to construct the simplicial set $\Eq(\tilde{U})$ via a construction that makes sense in any simplicially locally cartesian closed, simplicial model category $\cE$ that is right proper and  whose cofibrations are the monomorphisms (implying that the model category is also left proper). Recall from \eqref{eq:lcc}, that any morphism $f$ in a locally cartesian closed category gives rise to an adjoint triple $\Sigma_f \dashv f^* \dashv \Pi_f$ defined respectively by composition, pullback, and pushforward.  Since the cofibrations are the monomorphisms, pushforward along any map preserves trivial fibrations. Since the model category is right proper and the cofibrations are the monomorphisms, pushforward along a fibration also preserves fibrations \cite{GS}. Since the model category is simplicial, we may form the \textbf{fibered path space} factorization \eqref{eq:relative-path-space-factorization} of any fibration, factoring the relative diagonal through the fibered path space. As noted there, this construction has the following properties:

\begin{prop} For any fibration $p \colon E \fto B$ in a simplicial model category $\cE$ in which the cofibrations are the monomorphisms, the natural maps
  \begin{equation}\label{eq:rel-path-space-revisited}
    \begin{tikzcd}[sep=tiny]
      E \arrow[rrr, tail] \arrow[dddd, two heads, "p"'] \arrow[dr, tcofdashed] & & & E^{\Delta^1} \arrow[rrr] \arrow[dddd, two heads, "p^{\Delta^{1}}"] & & & E \times E \arrow[dddd, two heads, "p \times p"] \\ &  P_BE \arrow[dddl, two heads, dotted] \arrow[dddr, phantom, "\lrcorner" very near start]\arrow[urr, dotted] \arrow[dr, two heads, dashed ]
      \\ & & E \times_B E \arrow[ddll, two heads, dotted] \arrow[uurrrr, dotted]  \arrow[ddr, phantom, "\lrcorner" very near start]\\ ~
      \\ B \arrow[rrr, tail] & & ~ & B^{\Delta^1} \arrow[rrr] & & & B \times B
  \end{tikzcd} 
  \end{equation}
  give a factorization of the fibered diagonal map $(1,1) \colon E \to E \times_B E$ over $B$ as a trivial cofibration followed by a fibration. Moreover, when $\cE$ is simplicially cartesian closed, this construction is weakly stable over $B$, meaning that the pullback along any $f \colon A \to B$ is again such a factorization.
\end{prop}
\begin{proof} 
  Since $\cE$ is a simplicial model category and $p \colon E \fto B$ is a fibration, the maps depicted as fibrations in \eqref{eq:rel-path-space-revisited} are all fibrations. In particular, the endpoint evaluation map $(e_0,e_1) \colon  P_BE \twoheadrightarrow E\times_B E$ is a pullback of the Leibniz cotensor of the cofibration $\partial\Delta^1 \rightarrowtail \Delta^1$ and the fibration $p$ by Definition \ref{defn:simp-model-cat}. Similarly, the single endpoint evaluation maps $e_0, e_1 \colon P_BE \fwto E$ are pullbacks of the Leibniz cotensors of the trivial cofibrations $0,1 \colon \Delta^0 \cwto \Delta^1$ and the fibration $p$ and so are trivial fibrations. 

  By construction, the map $r \colon E \rightarrowtail P_BE$ is a section of both $e_0$ and $e_1$. Thus, it is a monomorphism and a weak equivalence by the 2-of-3 property. For the pullback stability of this construction, it is helpful to note that $P_BE \cong \Delta^1 \pitchfork_B E$ is the simplicial cotensor of $p \colon E \fto B$ by $\Delta^1$ in the slice $\cE_{/B}$. Simplicial cartesian closure demands that the adjunctions \eqref{eq:lcc} are simplicially enriched, and in particular that the right adjoint $f^* \colon \cE_{/B} \to \cE_{/A}$ preserves these simplicial cotensors.
\end{proof}

\begin{dig}
  In homotopy type theory, the object of equivalences between types in the universe $\univ$ is
  \begin{align*} \Eq(\univ) &\coloneqq \Sigma_{(A,B) : \univ\times\univ} \Equiv(A,B)\\
    \intertext{where} 
    \Equiv(A,B) &\coloneqq  \Sigma_{f : A \to B} \isEquiv(f)\\
    \intertext{where} 
  \type{isEquiv}(f) &\coloneqq \Sigma_{b : B} \type{isContr}(Pf_b)\\
  \intertext{where}
  \type{isContr}(C) &\coloneqq \Sigma_{c:C} \Pi_{x :C} c=_C x \\
  \intertext{and} Pf_b &\coloneqq \Sigma_{a:A} fa =_B b\end{align*} is the homotopy fiber of $f$ over $b$, using identity types, dependent pair types, and dependent function types.
  \end{dig}

We now establish the categorical analogues of each of these constructions  in a locally cartesian closed, right proper, simplicial model category whose cofibrations are the monomorphisms. A fibration $p \colon E \fto B$ is a weak equivalence if and only if it is a trivial fibration. Between cofibrant objects this is the case if and only if it admits a section $s \colon B \to E$ together with a homotopy $\gamma \colon sp \sim 1_E$ over $B$. Such data in particular chooses a center of contraction $s(b) \in E_b$ in the fiber over each $b \colon 1 \to B$ and also provides a contracting homotopy, proving the contractibility of the fibers of $p$.

The data of such a homotopy is encoded by a dashed lift 
\[
  \begin{tikzcd} & P_BE \arrow[d, two heads] \\ E \arrow[d, two heads, "p"'] \arrow[ur, dashed, bend left, "\gamma"] \arrow[r, "{(sp,1_E)}"]\arrow[dr, phantom, "\lrcorner" very near start] & E \times_B E \arrow[d, "\pi_1"]\\
    B \arrow[r, "s"'] & E
  \end{tikzcd} 
    \]
 Since the map $(sp,1_E)$ is a pullback of $s$ along $\pi_1$, this lift is equally encoded by a lift of $s \colon B \to E$ along $\Pi_{\pi_1}P_BE \fto E$. And thus a section of the composite $\Sigma_p \Pi_{\pi_1}P_BE \fto B$ provides the data of both the section $s$ and the homotopy $\gamma$. Thus we define
\[ \isContr_B(E) \coloneqq \Sigma_p \Pi_{\pi_1} P_BE \ \in\ \cE_{/B},\]
which is the $B$-indexed object internalizing the assertion that for each $b \in B$, the fiber $E_b$ is contractible. By construction:

\begin{lem}\label{lem:iscontr} A fibration $p \colon E \fto B$ is a trivial fibration if and only if $\isContr_B(E) \fto B$ has a section. \qed
\end{lem}

Now a map $f \colon E_1 \to E_2$ over $B$ is a weak equivalence if and only if its replacement by a weakly equivalent fibration over $E_2$ is a trivial fibration. This replacement is the map $p \colon P_bf \fto E_2$ defined by the pullback of the fibered path space of $E_2$ over $B$ in the following diagram.
\[
  \begin{tikzcd}
    E_1 \arrow[r, "f"] \arrow[dr, tcofdashed, "e"'] \arrow[ddr, bend right, "{\langle 1,f\rangle}"'] & E_2 \arrow[dr, tcofarrow] \\ & P_B f \arrow[r] \arrow[d, two heads, "{\langle q,p\rangle}"'] \arrow[dr, phantom, "\lrcorner" very near start] & P_BE_2 \arrow[d, two heads] \\&  E_1 \times_B E_2 \arrow[r, "f \times 1"'] & E_2 \times_B E_2
  \end{tikzcd}
\]
Note by construction, the map $q \colon P_Bf \fwto E_1$ is a trivial fibration, and thus its section $e \colon E_1 \cwto P_Bf$ is a trivial cofibration. 

By Lemma \ref{lem:iscontr}, the map $p \colon P_Bf \fto E_2$ is a trivial fibration if and only if $\isContr_{E_2}(P_Bf) \fto E_2$ has a section, or equivalently if and only if the pushforward along $p_2 \colon E_2 \fto B$ has a section over $B$. Thus we define
\[ \isEquiv_B(f) \coloneqq \Pi_{p_2} \isContr_{E_2}(P_Bf) \ \in\ \cE_{/B}.\]
By construction:

\begin{lem}\label{lem:isequiv}
  A map $f \colon E_1 \to E_2$ between fibrations is a weak equivalence if and only if $\isEquiv_B(f) \fto B$ has a section. \qed
\end{lem}

Finally the slice category $\cE_{/B}$ has an internal hom $\map_B(E_1,E_2)$. The counit $\epsilon \colon \map_B(E_1,E_2) \times_B E_1 \to E_2$ pulls back to define a map
\[ \mu \colon \map_B(E_1,E_2) \times_B E_1 \to \map_B(E_1,E_2)\times_B E_2\]
over $\map_B(E_1,E_2)$. The map $\mu$ is the ``universal map from $E_1$ to $E_2$ over $B$'' in the sense that the fiber of $\mu$ over an element $f \colon X \to \map_B(E_1,E_2)$ over $b \colon X \to B$ is a map  $f_b \colon b^*E_1 \to b^*E_2$ between the fibers over $b$.
\[
  \begin{tikzcd}[column sep=small, row sep=tiny]
    {b^*E_1} \arrow[drr, dotted, "f_b"'] \arrow[dr, phantom, "\lrcorner" very near start]\arrow[ddr, dotted, two heads] \arrow[rrr, dotted] & [-10pt] & [-10pt] & [-10pt] \map_{B}(E_1,E_2) \times_{B} (E_1) \arrow[rrr, two heads] \arrow[drd, two heads] \arrow[dr, phantom, "\lrcorner" very near start] \arrow[drrrrr, dashed, "\epsilon"] \arrow[drr, dashed, "\mu"'] &[-50pt]~ &[-80pt] ~ & [-30pt]  E_1 \arrow[ddr, two heads, "p_1"'] & [-10pt]~ & [-10pt]~  \\ [+1em]
    & ~& {B} \arrow[dl, dotted, two heads] \arrow[rrr, dotted] \arrow[dr, phantom, "\lrcorner" very near start] & ~& ~&  [-10pt]\map_{B}(E_1,E_2) \times_{B} (E_2) \arrow[rrr, crossing over] \arrow[dr, phantom, "\lrcorner" very near start] \arrow[dl, two heads] &~ &~ & E_2 \arrow[dl, two heads, "p_2"]\\ [+3em] &\Gamma \arrow[rrr, "f"] \arrow[rrrrrr, bend right=10, "{b}"'] &~& ~& \map_{B}(E_1,E_2) \arrow[rrr, two heads] & ~ & ~& B
    \end{tikzcd}
\]

Thus we define
\[ \Equiv_B(E_1,E_2) \coloneqq \Sigma_{\map_B(E_1,E_2)} \isEquiv_{\map_B(E_1,E_2)}(\mu) \ \in\ \cE_{/B}\] so that a section provides a map $E_1 \to E_2$ over $B$ that is an equivalence. More generally:
\begin{lem}\label{lem:Equiv}
   Given fibrations $p_1 \colon E_1 \fto B$ and $p_2 \colon E_2 \fto B$ between fibrations over $B$, any lift of $b \colon X \to B$ to $\Equiv_B(E_1,E_2)$ provides a map $b^*E_1 \to b^*E_2$ over $X$ that is an equivalence. 
\end{lem}

Finally, given a universe $\pi \colon \tilde{U} \to U$, the universal equivalence will define a fibration $\Eq(\tilde{U}) \twoheadrightarrow U \times U$, so that the pullback along a generalized element  recovers the space of equivalences just defined:
\[ \begin{tikzcd} \Equiv_B(E_1,E_2) \arrow[r] \arrow[d, two heads] \arrow[dr, phantom, "\lrcorner" very near start] & \Eq(\tilde{U}) \arrow[d, two heads] \\ B \arrow[r, "{(\name{p_1}, \name{p_2})}"'] & U \times U \end{tikzcd}\]
Projecting the map $(\name{p_1}, \name{p_2}) \colon B \to U \times U$ to each of the universe factors individually yields the pullbacks:
\[
\begin{tikzcd} E_1 \arrow[d, two heads, "p_1"'] \arrow[r] \arrow[dr, phantom, "\lrcorner" very near start] & \tilde{U} \times U \arrow[d, two heads, "\pi_1^*\pi"'] \arrow[dr, phantom, "\lrcorner" very near start] \arrow[r]& \tilde{U} \arrow[d, "\pi", two heads] & &E_2 \arrow[d, two heads, "p_2"'] \arrow[r] \arrow[dr, phantom, "\lrcorner" very near start] & U \times \tilde{U} \arrow[d, two heads, "\pi_2^*\pi"'] \arrow[r] \arrow[dr, phantom, "\lrcorner" very near start] & \tilde{U} \arrow[d, "\pi", two heads] \\
  \Gamma \arrow[r, "{(\name{p_1},\name{p_2})}"] \arrow[rr, bend right=10, "{\name{p_1}}"'] & U \times U \arrow[r, "\pi_1"] & U & & \Gamma \arrow[rr, bend right=10, "{\name{p_2}}"']\arrow[r, "{(\name{p_1},\name{p_2})}"] & U \times U \arrow[r, "\pi_2"] & U
\end{tikzcd}
\]
Thus, we define
\[ \Eq(\tilde{U}) \coloneqq \Equiv_{U \times U}(\pi_1^*\tilde{U}, \pi_2^*\tilde{U}) \ \in\ \cE_{/U \times U}.\]
By construction:
\begin{lem}\label{lem:Eq}
Given a map $(\name{p_1},\name{p_2}) \colon B \to U \times U$ classifying fibrations $p_1 \colon E_1 \fto B$ and $p_2 \colon E_2 \fto B$ over $B$, a lift to $\Eq(\tilde{U})$ provides an equivalence $E_1 \simeq E_2$ over $B$.
\[
  \begin{tikzcd}  & \Eq(\tilde{U}) \arrow[d, two heads] \\ B \arrow[ur, dashed, "\name{\epsilon}"] \arrow[r, "{(\name{p_1},\name{p_2})}"']& U \times U
  \end{tikzcd} \]
\end{lem}
\begin{proof}
  By Lemma \ref{lem:Equiv}, a lift of $(\name{p_1},\name{p_2}) \colon B \to U \times U$ to $\Eq(\tilde{U})$ provides map from $(\name{p_1},\name{p_2})^*\pi_1^*\tilde{U} \cong \name{p_1}^*\tilde{U} \cong E_1$ to $(\name{p_1},\name{p_2})^*\pi_2^*\tilde{U} \cong \name{p_2}^*\tilde{U} \cong E_2$ over $B$ that is an equivalence.
\end{proof}

In particular, there is a lift 
\[
  \begin{tikzcd} & \Eq(\tilde{U}) \arrow[d, two heads] \\ U \arrow[r, "{(1,1)}"'] \arrow[ur, dashed, "\name{\id}"] & U \times U
  \end{tikzcd}\]
classifying the identity equivalence $\tilde{U} \cong \tilde{U}$ over $U$. We may now state the model categorical version of Voevodksy's univalence axiom.

\begin{defn}\label{defn:univalence}
  The fibration $\pi \colon \tilde{U} \to U$ is \textbf{univalent} if the comparison map defined by any solution to the lifting problem 
\[
  \begin{tikzcd}[sep=3em] U \arrow[d, "\refl"', tcofarrow] \arrow[r, "\id"] & \Eq(\tilde{U}) \arrow[d, two heads] \\ U^{\Delta^1} \arrow[ur, dashed, "\text{id-to-equiv}" description] \arrow[r, two heads] & U \times U
  \end{tikzcd}
  \]
from the path object to the object of equivalences is an equivalence.
\end{defn}

By the 2-of-3 property, $\text{id-to-equiv} \colon U^{\Delta^1} \to \Eq(\tilde{U})$ is a weak equivalence if and only if $\id \colon U \to \Eq(\tilde{U})$ is a weak equivalence, which is the case if and only if either projection $\Eq(\tilde{U}) \to U$ is a trivial fibration:
\[
\begin{tikzcd} A \arrow[d, "i"', tail] \arrow[r, "\name{w}"] & \Eq(\tilde{U})\arrow[d, two heads, "\pi_2"] \\ B\arrow[ur, dashed]\arrow[r, "{\name{q_2}}"'] & U
\end{tikzcd}  
\] As with the fibrancy of the universe, this can be expressed as a property of the model category $\cE$ that does not refer explicitly to the universal Kan fibration $\pi \colon \tilde{U} \fto U$. 

  \begin{defn}\label{defn:EEP} The \textbf{equivalence extension property} states that given a cofibraton $i \colon A \cto B$, an equivalence $w \colon E_1 \wto E_2$ of small fibrations over $A$ and a fibration $D_2$ over $B$ that pulls back along $i$ to $E_2$, there exists a small fibration $D_1$ over $B$ and an equivalence $v \colon D_1 \wto D_2$ over $B$ that pulls back along $i$ to the equivalence $w$.
    \[
      \begin{tikzcd}[column sep=tiny, row sep=small]
        E_1 \arrow[drr, "w" near end, "\sim"'] \arrow[ddr, two heads, "p_1"']  \arrow[drrrrr, phantom, "\lrcorner" very near start] \arrow[rrr, dashed] & & & D_1 \arrow[ddr, two heads, dashed] \arrow[drr, "v", dashed, "\sim"'] \\ & & E_2 \arrow[drr, phantom, "\lrcorner" very near start] \arrow[dl, two heads, "p_2"] \arrow[rrr] & & & D_2 \arrow[dl, two heads, "q_2"] \\ & A \arrow[rrr, tail, "i"'] & & & B
      \end{tikzcd}\]
  \end{defn}

As with Lemma \ref{lem:FEP-fibrant}, there is a slightly delicate argument necessary to prove that the equivalence extension property implies univalence of the fibration $\pi \colon \tilde{U} \to U$ as stated in Definition \ref{defn:univalence}. Suppose the equivalence extension property holds and consider a commutative square
  \[
    \begin{tikzcd}
    A \arrow[r, "\name{w}"] \arrow[d, tail, "i"'] & \Eq(\tilde{U}) \arrow[d, two heads, "\pi_2"] \\ B \arrow[r, "\name{q_2}"'] \arrow[ur, dashed, "v"'] & U    
    \end{tikzcd}
  \]
By the equivalence extension property, $w \colon E_1 \simeq E_2$ extends to an equivalence $v \colon D_1 \simeq D_2$ between fibrations over $B$. By realignment, there is a classifying map for $D_1$ extending the classifying map for $E_1$.
\[
  \begin{tikzcd}
    E_1 \arrow[dd, two heads, "p_1"'] \arrow[rr] \arrow[drr, phantom, "\lrcorner" very near start] \arrow[dddr, phantom, "\lrcorner" very near start]\arrow[dr] & & \tilde{U} \arrow[dd, two heads, "\pi"] \\ & D_1  \arrow[dr, phantom, "\lrcorner" very near start]\arrow[ur, dashed] & ~ \\ A \arrow[dr, tail, "i"'] \arrow[rr, "{\name{p_1}}" near end] & & U \\ & B \arrow[ur, dashed, "{\name{q_1}}"'] \arrow[from=uu, two heads, crossing over, "q_1"' near start]
  \end{tikzcd}  
\]
This defines a new lifting problem as below-left together with the data of an equivalence $v \colon D_1 \simeq D_2$ over $B$ extending $w \colon E_1 \simeq E_2$. Unraveling the definitions, this gives us the lifting problem below-center and then below-right: 
\[
  \begin{tikzcd}
 A \arrow[d, "i"', tail] \arrow[r, "\name{w}"] & \Eq(\tilde{U}) \arrow[d, two heads] & A \arrow[d, "i"', tail] \arrow[r, "w= i^*v"] & \Equiv_B(D_1,D_2) \arrow[d, two heads] & A \arrow[d, "i"', tail] \arrow[r, "{\name{w}}"] & \isEquiv_B(v) \arrow[d, two heads]\\ B \arrow[ur, dashed, "{\name{v}}"'] \arrow[r, "{(\name{q_1}, \name{q_2})}"'] & U \times U    & B \arrow[ur, dashed] \arrow[r, "{\name{v}}"'] & \map_B(D_1,D_2) & B \arrow[ur, dashed] \arrow[r, equals] & B
  \end{tikzcd}
\]
Since $v$ is a weak equivalence, the map $\isEquiv_B(v) \fto B$ is a trivial fibration,\footnote{We proved a bit less; see \cite[4.3]{shulman-reedy} for more.} so there exists a lift in the right-hand square which induces one in the middle square and then in the left-hand square. In particular, this final lift solves the original lifting problem, proving univalence.

Thus, we have reduced the verification of the univalence axiom to proving the equivalence extension property for Kan fibrations of simplicial sets. We define $D_1$ and $v$ by the pullback
\[
\begin{tikzcd} D_1 \arrow[d, dashed, "v"'] \arrow[r, dashed] \arrow[dr, phantom, "\lrcorner" very near start] & \Pi_i E_1 \arrow[d, "\Pi_i w"] \\ D_2 \arrow[r, "\eta"'] & \Pi_i E_2
\end{tikzcd}
\]
and leave the rest of the verification to the literature \cite{klv,shulman-reedy,sattler}.

\begin{rmk}
Note that fibrancy of $U$ was not required to prove the equivalence extension property, and in fact follows from it by an argument due to Stenzel \cite[2.4.3]{stenzel}. We must construct a lift
\[
  \begin{tikzcd}
    A \arrow[d, "i"', tcofarrow] \arrow[r, "{\name{p}}"] & U \\ B \arrow[ur, dashed, "{\name{q}}"']
  \end{tikzcd}
\]
that is, we must extend a fibration as indicated:
\[
  \begin{tikzcd}
    E_1 \arrow[d, two heads, "p_1"'] \arrow[r, dashed] \arrow[dr, phantom, "\lrcorner" very near start] & D_1 \arrow[d, two heads, dashed, "q_1"] \\ A \arrow[r, tail, tcofarrow, "j"'] & B
  \end{tikzcd}
\]
Factor the composite $j \cdot p_1$ into an acyclic cofibration followed by a fibration and pull the fibration back:
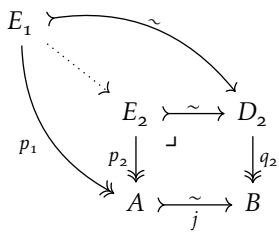
\[
  \begin{tikzcd} E_1 \arrow[dr, dotted] \arrow[ddr, bend right, "p_1"', two heads] \arrow[drr, bend left, tail, tcofarrow] \\ & E_2 \arrow[d, two heads, "p_2"'] \arrow[r, tcofarrow] \arrow[dr, phantom, "\lrcorner" very near start] & D_2 \arrow[d, two heads, "q_2"] \\ & A \arrow[r, tail, tcofarrow, "j"'] & B
  \end{tikzcd}
\]
By right properness, $j$ pulls back to a weak equivalence and thus the induced map to the pullback is an equivalence between fibrations over $A$. Now the equivalence extension property can be used to extend the fibration $E_1$ as desired.
\end{rmk}

\part*{Lecture III : All \texorpdfstring{$\infty$}{infinity}-toposes have strict univalent universes}

\section{Model categorical presentations of \texorpdfstring{$\infty$}{infinity}-topoi}\label{sec:topoi}

We have sketched, at least in one example, how homotopy type theory can be interpreted in at least one category: namely the category of simplicial sets. This interpretation takes the form of a functor whose domain is the category of contexts and whose codomain is a contextual category built from the category of simplicial sets using a universe. Note, in particular, that this interpretation is a functor between 1-categories that preserves various categorical structures, corresponding to the logical rules of homotopy type theory, on the nose.

As we now explain, Voevodsky's result can be understood as giving an interpretation of homotopy type theory in the \emph{$\infty$-topos of spaces}. More precisely, the 1-category of simplicial sets is equipped with the structure reviewed in \S\ref{ssec:quillen} that make it a simplicial model category, and via this structure it can be understood to \emph{present} an $\infty$-category, namely the $\infty$-category of spaces.\footnote{See \cite{riehl-homotopical} for a survey of the connection between model categories and $\infty$-categories.} Once we have adapted to this perspective on the category of simplicial sets, we  introduce general $\infty$-topoi via a definition that highlights their presentations by simplicial model categories.

The notion of $\infty$-topos we consider, developed by Lurie \cite{lurie-topos} and Rezk \cite{rezk} is analogous to Grothendieck 1-toposes. In fact, every Grothendieck 1-topos arises as the full subcategory of $0$-truncated objects in an $\infty$-topos. But the world of $\infty$-topoi is considerably richer than their 1-dimensional analogues: there are inequivalent $\infty$-topoi whose 0-truncations agree.

\begin{cav}\label{cav:infinity}
This is not at all how I ordinarily think about $\infty$-categories; for that see \cite{RV}. So if you want to object that model categorical presentations of $\infty$-categories are somewhat beside the point---like (gasp!) choosing a basis for your vector space---I agree. However, these presentations are essential to the construction of the interpretation of the syntax of homotopy type theory, as we shall see in \S\ref{sec:topos-semantics}, which is why they will be the focus here.
\end{cav}

\subsection{Simplicial model categories as \texorpdfstring{$\infty$}{infinity}-categories}

Here, following Lurie, we use ``$\infty$-\textbf{categories}'' as a nickname for \textbf{$(\infty,1)$-categories}, which are meant to be categories weakly enriched in $\infty$-groupoids aka ``spaces.'' It turns out that every $\infty$-category can be presented by a category that is strictly enriched in spaces, where by ``spaces'' we mean simplicial sets since the category of simplicial sets is so much better behaved. Though not all simplicially sets are equally ``space-like,'' so it would be more precise to say that an $\infty$-category is a category enriched in Kan complexes \cite[7.8]{bergner}, \cite[2.2.5.1]{lurie-topos}.

An abundant source of categories enriched in Kan complexes is given by simplicial model categories. The definition given in \S\ref{ssec:quillen} is equivalently restated as follows: a model category $\cE$ is a \textbf{simplicial model category} if $\cE$ is enriched, tensored, and cotensored over simplicial sets---and in particular has a mapping space bifunctor $\Map \colon \cE^\op \times \cE \to \Set$---and if moreover for any cofibration $\ell \colon A \cto B$ and fibration $r \colon X \fto Y$ in $\cE$, the Leibniz mapping space map 
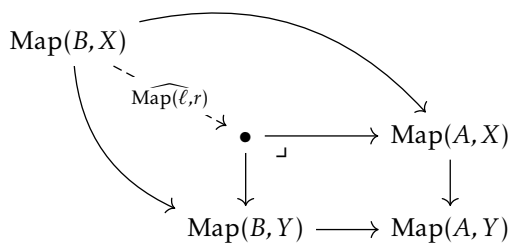
\[
\begin{tikzcd} 
 {\Map(B,X)} \arrow[dr, dashed, "{\widehat{\Map(\ell,r)}}" description] \arrow[ddr, bend right] \arrow[drr, bend left] \\   ~ & [-1em] \bullet \arrow[r] \arrow[d] \arrow[dr, phantom, "\lrcorner" very near start] &{\Map(A,X)} \arrow[d] \\ & {\Map(B,Y)} \arrow[r] & {\Map(A,Y)}
\end{tikzcd}
\] 
is a Kan fibration of simplicial sets that is a trivial fibration if either $\ell$ or $r$ is a weak equivalence. As a simplicially enriched category, $\cE$ is of course simplicially enriched, but as a simplicial model category the full subcategory of \emph{fibrant} and \emph{cofibrant} objects is enriched over Kan complexes. This gives the first sense in which a simplicial model category presents an $\infty$-category.

It might seem as if we are throwing away the non-fibrant and non-cofibrant objects when we pass to the $\infty$-category associated to a simplicial model category and there is a sense in which we are, but also a sense in which we are not. To explain, first note that any model category $\cE$ has a \textbf{homotopy category} $\cE[\we^{-1}]$ defined by formally inverting the weak equivalences; the model category axioms guarantee that this process is well-behaved. The homotopy category is equivalent to the category obtained by taking the path components of each space in the Kan complex enriched category of fibrant-cofibrant objects. Via this equivalence, the missing objects can be understood to be isomorphic, at the level of the homotopy category and thus also the $\infty$-category, to the fibrant-cofibrant objects that are still present.

The $\infty$-categories that are presented by simplicial model categories are special in one respect: they are complete and cocomplete as $\infty$-categories. Limits and colimits in the $\infty$-category are presented by \emph{homotopy limits} and \emph{homotopy colimits} definable in the simplicial model category \cite[4.2.4.1]{lurie-topos}. These are constructed via a two-step procedure. Given a diagram in a model category, it first needs to be ``fattened-up,'' by replacing certain maps by cofibrations if the aim is to form a homotopy colimit or by fibrations if the aim is to form a homotopy limit. Then the ordinary colimit or limit of the fattened-up diagram models the homotopy colimit or homotopy limit \cite[\S5]{riehl-homotopical}.

\begin{ex}\label{ex:sset-as-model-cat} As noted in \S\ref{ssec:quillen}, the category of simplicial sets defines a simplicial model category:
  \begin{itemize}
    \item The cofibrations are the monomorphisms, the weak equivalences are the weak homotopy equivalences, and the fibrations are the Kan fibrations. All objects are cofibrant, while the fibrant objects are the Kan complexes.
    \item As a complete and cocomplete cartesian closed category, simplicial sets is enriched, tensored, and cotensored over itself. Moreover, the simplicial enrichment is compatible with the model structure in the sense described in Remark \ref{rmk:sm7}.
  \end{itemize}
The associated $\infty$-category is the $\infty$-category of spaces, typically denoted by $\cS$.
\end{ex}

The $\infty$-category of spaces plays a role for $\infty$-category theory that is analogous to the role played by the category of sets in 1-category theory. There is an important sense, though, in which the category of simplicial sets is better behaved than the category of sets that is clarified by the following definition:

\begin{defn}[{\cite[6.5]{rezk}}]\label{defn:descent}
  A simplicial model category $\cE$ satisfies \textbf{descent} if:
  \begin{enumerate}
    \item Homotopy colimits are stable under homotopy pullback: for any $X = \hocolim_i X_i$ and $f \colon Y \to X$, the map $\hocolim_i (Y \tilde{\times}_X X_i) \to Y$ from the homotopy colimit of the homotopy pullbacks is a weak equivalence.
    \item\label{itm:descent} For a natural transformation between diagrams whose naturality squares are homotopy pullbacks, the naturality squares formed by the homotopy colimit cones are also homotopy pullbacks.\footnote{This can be understood to mean that the canonical comparison map from the initial vertex in the square to the homotopy pullback is a weak equivalence.}
  \end{enumerate}
\end{defn}

\begin{rmk} Any $\infty$-category $\cE$ with pullbacks has a (contravariant) \textbf{core of self-indexing} functor
\[ \begin{tikzcd}[row sep=tiny]  \cE^\op \arrow[r, "\EE"] & \infty\text{-}\GPD  \\ B \arrow[r, maps to] & (\cE_{/B})^\cong\arrow[d, "f^*"] \\ [7pt] A \arrow[u, "f"] \arrow[r, maps to] & (\cE_{/A})^\cong \end{tikzcd}\] 
The $\infty$-category satisfies descent just when this functor carries colimits in $\cE$ to limits in the $\infty$-category of $\infty$-groupoids.
\end{rmk}

While the category of sets satisfies a limited form of descent---for certain, but not all, colimit diagrams---the $\infty$-category of spaces satisfies descent for colimit diagrams of all shapes \cite[2.1-2]{rezk}.

\subsection{Simplicial model categories of simplicial presheaves}\label{ssec:simp-presheaves}

The descent property is very convenient, so it is natural to wonder what other $\infty$-categories might satisfy descent. A natural idea is to consider diagram $\infty$-categories $\cS^{\cD^\op}$ valued in the $\infty$-category of spaces. Since limits and colimits in diagram $\infty$-categories are defined pointwise in terms of limits and colimits in the base $\infty$-category, one would expect the properties of Definition \ref{defn:descent} to be inherited by $\cS^{\cD^\op}$.

This argument would work just fine had we developed the theory of limits and colimits natively at the level of $\infty$-categories, which is certainly possible to do. However, we have defined limits and colimits in a bicomplete $\infty$-category by using a presentation by a simplicial model category. Thus, in order to describe limits and colimits in a diagram $\infty$-category we need to find a suitable simplicial model category presentation.

Let $\cE$ be any simplicial model category and let $\cD$ be a small simplicially enriched category.\footnote{Since we are mapping out of $\cD$ and into $\cE$ it is not important whether the hom-spaces in $\cD$ are Kan complexes.} We can certainly define a simplicial model category structure on $\cE^{\ob\cD} \cong \prod_{\ob\cD}\cE$ by defining everything ``pointwise.'' It is natural to hope that the simplicial category $\cE^{\cD^\op}$ of simplicially enriched functors inherits a similar pointwise-defined simplicial model structure, but there are two (related) problems:
\begin{enumerate}
  \item The pointwise-defined cofibrations fail to lift against the pointwise-defined trivial fibrations, and similarly the pointwise-defined trivial cofibrations fail to lift against the pointwise-defined fibrations.\footnote{This failure will be explained in more detail in \S\ref{ssec:cobar}.}
  \item The simplicial hom-space
  \[ \Map_{\cE^{\cD^\op}}(X,Y) \coloneqq \int_{d \in \cD} \Map_\cE(Xd,Yd)\]
  between pointwise fibrant-cofibrant diagrams may fail to be a Kan complex.
\end{enumerate}
It turns out that a solution to the first problem will solve the second as well. In fact, we present two solutions, which are in some sense equivalent. To describe them note firstly that the forgetful functor
\[
\begin{tikzcd}[sep=large] \cE^{\cD^\op} \arrow[r, "U" description] & \cE^{\ob\cD} \arrow[l, bend left, "R", "\bot"'] \arrow[l, bend right, "L"', "\bot"]
\end{tikzcd}
\]
admits left and right adjoints, defined by simplicially enriched left and right Kan extensions.

\begin{defn} For any functor $U \colon \cM \to \cN$ where $\cN$ is a model category define a morphism in $\cM$ to be:
  \begin{itemize}
  \item a $U$-cofibration, $U$-weak equivalence, or $U$-fibration just when its image under $U$ belongs to those respective classes
  \item a \textbf{projective cofibration/projective trivial cofibration} when it has the left lifting property for $U$-trivial fibrations/$U$-fibrations
  \item a \textbf{injective fibration/injective trivial fibration} when it has the right lifting property for $U$-trivial co\-fib\-rations/$U$-cofibrations
  \end{itemize}
  \end{defn}

When the $U$-weak equivalences, $U$-fibrations, and projective cofibrations form a model structure, it is called the \textbf{projective model structure}. Dually, when the $U$-weak equivalences, $U$-cofibrations, and injective fibrations form a model structure, it is called the \textbf{injective model structure}. The following result records one set of conditions that guarantee that these model structures exist:

\begin{prop}[{\cite[8.2]{shulman}}]\label{prop:proj-inj} Given a model category $\cN$, $U \colon \cM \to \cN$ with both adjoints $L \dashv U \dashv R$ so that $UL \dashv UR$ is Quillen, then if $\cN$ is combinatorial\footnote{This is a technical set-theoretical condition that demands that the underlying category is \emph{locally presentable} and the model structure is \emph{cofibrantly generated}.} and $\cM$ is locally presentable, both projective and injective model structures exist and are combinatorial. Moreover:
\begin{itemize}
\item The projective and injective models structures are left or right proper if $\cN$ is.
\item If $\cN$ is a simplicial model category and the adjunctions are simplicial adjunctions then $\cM$ is a simplicial model category with either the projective or injective model structures.
\item Projective cofibrations in particular define $U$-cofibrations and projective fibrations in particular define $U$-fibrations.
\end{itemize}
\end{prop}

In particular, Proposition \ref{prop:proj-inj} can be applied to deduce the existence of both the projective and injective model structures on $\sSet^{\cD^\op}$ for any small simplicial category $\cD$; see \cite[A.3.3.9]{lurie-topos} or \cite{moser}. And it turns out that both of these present the $\infty$-category $\cS^{\cD^\op}$ of presheaves: the identity adjunction
\[
  \begin{tikzcd} \sSet^{\cD^\op}_\proj \arrow[r, bend left, "="] \arrow[r, phantom, "\bot"] & \sSet^{\cD^\op}_\inj \arrow[l, bend left, "="]
  \end{tikzcd}
\]
defines a \textbf{Quillen adjunction} between the projective and injective model structures, meaning that the left adjoint preserves the left classes and the right adjoint preserves the right classes. This Quillen adjunction descends to an adjoint equivalence between the homotopy categories of the model categories and hence is a \textbf{Quillen equivalence}. It follows that both model structures present the $\infty$-topos $\cS^{\cD^\op}$.

Finally, we observe that the simplicial model categories $\cS^{\cD^\op}$, with either the projective or injective model structures, satisfy descent. To see this, note that the forgetful functor $U \colon \sSet^{\cD^\op} \to \sSet^{\ob\cD}$ preserves homotopy limits and homotopy colimits and creates weak equivalences.

\subsection{Rezk model toposes}

The simplicial model category $\sSet^{\cD^\op}$, with either the projective or injective model structure is a prototypical example of a \textbf{model topos} as defined by \cite{rezk} inspired by previous work of Simpson \cite{simpson} and To\"{e}n and Vezzosi \cite{TV}. Rezk proves a Giraud-style theorem characterizing model toposes that we take as the definition.

\begin{defn}[{\cite[6.9]{rezk}}] A model category $\cE$ is a \textbf{model topos} if and only if
  \begin{enumerate}
    \item\label{itm:topos-presentation} $\cE$ admits a \textbf{small presentation}.
    \item\label{itm:topos-descent} $\cE$ has descent.
  \end{enumerate}
\end{defn}

The meaning of \textbf{small presentation} for a model category is developed in work of Dugger \cite{dugger-universal,dugger-combinatorial} where it is characterized more precisely than we require here. The essential point\footnote{Dugger's small presentations involve localizations of the \emph{projective} model structure rather than the \emph{injective} model structure. Since the projective and injective model structures are Quillen equivalence and the process of localization preserves this Quillen equivalence, up to Quillen equivalence it makes no difference which model structure is used. The cost of using the injective model structure rather than the projective one is that there is no longer a direct Quillen equivalence to the model category $\cE$ being presented but rather a zig-zag of Quillen equivalences pointing in opposite directions. For our purposes here, however, this is immaterial, and we prefer the injective model structure because its cofibrations are exactly the monomorphisms, a hypothesis used in \S\ref{ssec:univalence}.} is a model category $\cE$ admits a small presentation just when it is Quillen equivalent to a model category of the form $\sSet^{\cD^\op}_S$, that is, to a \emph{localization} of the injective model structure on simplicial presheaves on a small simplicial category $\cD$ by a set of maps $S$.  The category $\cD$ can be thought of as providing ``generators'' for $\cE$ with the set $S$ gives the ``relations.'' 

Here ``localization'' refers to \emph{left Bousfield localization}, a process which ``formally inverts'' the maps in $S$ by turning them into weak equivalences while restricting the class of fibrations; this process does not change the underlying category or the class of cofibrations.

  
The descent condition in \eqref{itm:topos-descent} puts a further restriction on the model categories that satisfy \eqref{itm:topos-presentation}. Rezk proves that a localized model category of the form $\sSet^{\cD^\op}_S$ satisfies descent just when the left adjoint of the canonical Quillen adjunction
\[
  \begin{tikzcd} 
    \sSet^{\cD^\op} \arrow[r, bend left, "="] \arrow[r, phantom, "\bot"] & \sSet^{\cD^\op}_S \arrow[l, bend left, "="]
  \end{tikzcd}
\]
is \textbf{left exact}, meaning that it preserves homotopy pullbacks, and hence all finite homotopy limits. By composing zig-zags of Quillen equivalences:

\begin{thm}[{\cite[6.1]{rezk}}]\label{thm:model-topos} A model category is a \textbf{model topos} if and only if it is Quillen equivalent to a left exact localization $\sSet^{\cD^\op}_S$ of an injective model category of simplicial presheaves indexed by a small simplicial category $\cD$ at a set of maps $S$.
\end{thm}

Dugger proves that all combinatorial model categories have a presentation \cite{dugger-combinatorial}, so model toposes can be thought of as combinatorial simplicial model categories of simplicial sheaves, where the data $(\cD,S)$ of Theorem \ref{thm:model-topos} together define the \textbf{model site}. Model toposes are simplicial model categories that present $\infty$-\textbf{toposes}, which Lurie defines to be accessible left exact reflective sub-$\infty$-categories of an $\infty$-category $\cS^{\cD^\op}$ of presheaves and then characterizes in many different ways \cite[6.1]{lurie-topos}.

\begin{ex} The simplicial model category of simplicial sets described in Example \ref{ex:sset-as-model-cat} a model topos.
\end{ex}

\begin{ex} Simplicial presheaves on any small simplicial category $\cD$ defines a model topos with either the projective or injective model structure.
\end{ex}

\begin{ex} Whenever $\cC$ has finite colimits and a terminal object, Goodwillie's $n$-excisive approximation defines a left exact localization of $\sSet^\cC$, yielding a model topos of $n$-excisive functors.
\end{ex}

The analogue of the truncation levels in homotopy type theory is defined as follows:

\begin{defn} A space $X$ is $k$-\textbf{truncated} ($k \geq -1$) if $\pi_q(X,x)$ is trivial for all $q > k$ and for all choices of basepoint. We extend to say $-2$-truncated for spaces that are contractible.
\end{defn}

This definition then extends to an object $X$ in a simplicial model category by applying it the mapping spaces $\Map(Y,X)$ for all $X$. For any simplicial model category $\cE$ write $\tau_k\cE$ for the full subcategory of $k$-truncated objects. 

\begin{prop} If $\cE$ is a model topos then $\tau_0\cE$ is a Grothendieck 1-topos.
\end{prop}

In fact every Grothendieck 1-topos is a 0-truncation of a model topos, much like every locale arises as the $-1$-truncated objects in a Grothendieck 1-topos. 

\begin{ex}[{\cite[11.2]{rezk}}] Any Grothendieck site, the data of which is given by a small 1-category $\cC$ with a collection of covering sieves $S \cto \yo_c \in \Set^{\cC^\op}$, may be regarded as a model site by regarding sets as discrete spaces. Then the model category defined by localizing simplicial presheaves  $\sSet^{\cC^\op}$ at the set of covering sieves is a model topos whose $0$-truncation is the 1-category  of sieves on the 1-site.
\end{ex}

The reason that the indexing category in a model topos is allowed to be simplicially enriched, and not just a 1-category, is to allow model topoi whose $k$-truncations differ even though their 0-truncations are the same. See \cite[\S 11]{rezk} for discussion and many more examples.

\section{The \texorpdfstring{$\infty$}{infinity}-topos semantics of homotopy type theory}\label{sec:topos-semantics}

We now turn to Shulman's paper \cite{shulman}.

\subsection{Type-theoretic model toposes}

Shulman introduces the following axiomatization to collect the hypotheses he requires. 

\begin{defn}\label{defn:ttmt} A \textbf{type-theoretic model topos} is a model category such that:
  \begin{enumerate}
    \item The underlying category is a Grothendieck 1-toposes.
    \item The model category is a right proper simplicial Cisinski model category: the cofibrations are the monomorphisms and the model structure is proper, simplicial, and combinatorial.
    \item The underlying category is simplicially locally cartesian closed, which amounts to the additional requirement that pullbacks preserve simplicial tensors.
    \item\label{itm:notion-of-fibered} There is a locally representable and relatively acyclic notion of fibered structure $\FF$ so that $|\FF|$ is the class of all fibrations.
  \end{enumerate}
\end{defn}

The final axiom, which will be unpacked in \S\ref{ssec:fibered}, ensures that there is a regular cardinal $\lambda$ so that $\cE$ has a fibrant univalent universes of relatively $\kappa$-presentable fibrations for any regular cardinal $\kappa \triangleright \lambda$.\footnote{Again, we'll largely sweep size issues on the rug, but here are the definitions: an object is $\kappa$-presentable if its covariant representable preserves $\kappa$-filtered colimits. A morphism is relatively $\kappa$-presentable if its pullback over any $\kappa$-presentable object is $\kappa$-presentable. Here $\kappa \triangleright\lambda$ means that $\kappa > \lambda$ is so that any $\lambda$-presentable category is $\kappa$-presentable.}

\begin{thm}\label{thm:ttmt-interpretation} For any type-theoretic model topos $\cE$, there is a regular cardinal $\lambda$ so that $\cE$ interprets Martin-L\"of type theory with the following structure:
  \begin{enumerate}
  \item $\Sigma$-types, a unit type, $\Pi$-types with function extensionality, identity types, and binary sum types;
  \item the empty type, the natural numbers type, the circle type, the sphere types, and other specific ``cell complexes'';
  \item as many universe types as there are inaccessible cardinals larger than $\lambda$, closed under the type formers (i), containing the types (ii), and satisfying the univalence axiom;
  \item\label{itm:rec-hits} $W$-types, pushout types, truncations, localizations, James constructions, and many other recursive higher inductive types.
  \end{enumerate}
  \end{thm}
  
It is not yet known whether the universe is closed under the higher inductive types listed in \eqref{itm:rec-hits}.

\begin{dig}
  One needs a coherence theorem to strictify $\cE$ into an actual model of type theory. This achieved through the ``local universes'' technique of Lumsdaine-Warren \cite{lumsdaine-warren}. Lumsdaine and Warren observe that universes are used for two distinct purposes in the simplicial model: ``firstly to obtain coherence of the model; and secondly to become type-theoretic universes within the model.'' They observe ``It turned out that not only may the two aspects be entirely disentangled, but moreover, the coherence construction may be modified to work without a universe.''
  
They work with \emph{comprehension categories} \cite{jacobs} as a categorical model of type theory, which is a suitably structured fibration of categories---in which the base category has the contexts as objects and the substitutions as morphisms and the fibers encode dependent types in that context---together with further algebraic structure implementing the logical rules. Direct models of syntax are given by split comprehension categories with strictly stable logical structure. They replace a comprehension category by a split comprehension category using ``delayed substitution.'' The split comprehension category has the same category of contexts as the original one but the types in context $\Gamma$ are formally defined to be maps $\name{A} \colon \Gamma \to U_A$ valued in a ``local universe,'' which is just another context, together with a type $E_A$ in that context. The triple $(\name{A},U_A,E_A)$ may be regarded as a surrogate for the pullback $[A] \coloneqq E_A[\name{A}]$. The advantage of using local universes is that they are allowed to vary when constructing the logical operations. For instance, for sums one can take $V_{A+B} \coloneqq V_A \times V_B$ with $E_{A+B} \coloneqq E_A[\pi_1] +_{V_{A+B}} E_B[\pi_2]$, with $\name{A+B} \colon \Gamma \to V_{A+B}$ the evident map. Since substitution corresponds to precomposition ``on the left'' but this logical stuff happens ``on the right,'' such constructions are strictly stable.

The main theorem is that a full comprehension category whose underlying category has certain products and pushforwards has an equivalent full split comprehension category so that if the comprehension category has weakly stable categorical operations corresponding to the various type theoretic operations, then the equivalent split comprehension category has strictly stable versions of the same. In an appendix \cite[\S A]{shulman}, Shulman adapts this coherence theorem to Awodey's natural models \cite{awodey} and extends the coherence theorem to universes of types.
\end{dig}

Theorem \ref{thm:ttmt-interpretation} tells us homotopy type theory can be interpreted into any type-theoretic model topos. It remains to relate type-theoretic model toposes to Rezk's model toposes.

\begin{thm}[{\cite[6.4]{shulman}}]\label{thm:ttmt-descent} A type-theoretic model topos has descent.
\end{thm}
\begin{proof}
It is straightforward to argue from the model category axioms that homotopy colimits in a type-theoretic model topos are stable under homotopy pullback. A more delicate argument is required for the second condition of Definition \ref{defn:descent}, involving both the local representability and the relative acyclicity of the notion of fibered structure.
\end{proof}

With this result in hand, it follows easily that:

\begin{cor} Every type-theoretic model topos is a model topos.
\end{cor}
\begin{proof} Since a type-theoretic model topos is combinatorial, it has a small presentation by \cite{dugger-combinatorial} and Theorem \ref{thm:ttmt-descent} demonstrates that it has descent.
\end{proof}

To prove that homotopy type theory admits a model in any $\infty$-topos it remains to show that any model topos is Quillen equivalent to a type-theoretic one. Before we can approach this problem, we must first explain the final axiom \eqref{itm:notion-of-fibered} of Definition \ref{defn:ttmt}.

\subsection{Notions of fibered structure}\label{ssec:fibered}

We have seen one example of a notion of fibered structure, which we recall before giving the general definitions.

In a model category $\cE$, the (contravariant) \textbf{core of self-indexing} defines a groupoid-valued pseudofunctor as below-left:
\[ \begin{tikzcd}[row sep=tiny]  \cE^\op \arrow[r, "\EE"] & \GPD  & & \cE^\op \arrow[r, "\FF"] & \GPD  \\ B \arrow[r, maps to] & (\cE_{/B})^\cong\arrow[d, "f^*"] & &  B \arrow[r, maps to] & (\sF_{/B})^\cong\arrow[d, "f^*"] \\ [7pt] A \arrow[u, "f"] \arrow[r, maps to] & (\cE_{/A})^\cong & & A \arrow[u, "f"] \arrow[r, maps to] & (\sF_{/A})^\cong\end{tikzcd}\]  Since the class of fibrations $\fib$ in a model category are always pullback stable, they similarly assemble into a pseudofunctor as above right.

The inclusion $\iota \colon \FF\hookrightarrow\EE$ defines a 1-cell in the 2-category $\PSH(\cE)$ of groupoid-valued pseudofunctors, pseudonatural transformations, and modifications with the following properties: 
\begin{enumerate}
  \item $\iota$ is a strict natural transformation, and
  \item each component of $\iota$ is a discrete fibration of groupoids.\footnote{Any componentwise discrete fibration in $\PSH(\cE)$ is automatically an \textbf{internal discrete fibration} in the 2-category $\PSH(\cE)$, meaning any 2-cell with codomain $\EE$ whose codomain 1-cell factors through $\iota$ has a unique lift with specified codomain 1-cell.}
\end{enumerate}
In what follows, we continue to write $\EE$ for the core of self indexing but allow $\FF$ to be any notion of fibered structure, satisfying the following definition:

\begin{defn} A \textbf{notion of fibered structure} on $\cE$ is a strict  discrete fibration $\phi \colon \FF \to \EE$ in $\PSH(\cE)$ that has small fibers.
\end{defn}

Recall the fibers of a discrete fibration are sets. Thus $\phi \colon \FF \to \EE$ equips each $f \colon X \to Y$ in $\EE(Y)$ with a set of $\FF$-\textbf{structures} on $f$ so that pullbacks induce functions on $\FF$-structures. We refer to a morphism with a chosen $\FF$-structure as an $\FF$-\textbf{algebra}, while an $\FF$-\textbf{morphism} is a pullback square 
\[ \begin{tikzcd} X' \arrow[d, "f'"'] \arrow[r] \arrow[dr, phantom, "\lrcorner" very near start] & X \arrow[d, "f"] \\ Y' \arrow[r] & Y \end{tikzcd}\] in which the $\FF$-algebra structure on $f'$ is induced from the $\FF$-algebra structure on $f$.

Shulman gives a long list of examples in \cite[\S 3]{shulman}, but we will only need a few.

\begin{ex} If $\FF \hookrightarrow \EE$ is the inclusion of a subfunctor, then $\FF$ is just a pullback stable class of morphisms in $\cE$. We call this a \textbf{full notion of fibered structure}. \end{ex}

In particular, if $\cE$ is a model category, the fibrations define a notion of fibered structure in this way.

\begin{ex}\label{ex:fun-fact} A \textbf{functorial factorization} is a functor $\vec{E} \colon \cE^\2 \to \cE^\3$ that defines a section of the composition functor $\circ \colon \cE^\3 \to \cE^\2$. On morphisms it carries a commutative square as below-left to a commutative rectangle as below-right:
  \[\begin{tikzcd}
    X' \arrow[d, "f'"'] \arrow[r, "x"] & X \arrow[d, "f"] \\ Y' \arrow[r, "y"'] & Y 
  \end{tikzcd} \quad = \quad
  \begin{tikzcd} X' \arrow[r, "x"] \arrow[d, "Lf'"'] & X \arrow[d, "Lf"] \\ Ef' \arrow[r, "{E(x,y)}"] \arrow[d, "{Rf'}"'] & Ef \arrow[d, "Rf"] \\ Y' \arrow[r, "y"'] & Y
  \end{tikzcd}
       \] 
       In particular, a functorial factorization gives rise to a pair of endofunctors $L,R \colon \cE^\2 \to \cE^\2$, which send a map to its left and right factors, respectively. We define an $\RR_{{E}}$-structure on $f$ to be a lift
       \[\begin{tikzcd} X \arrow[d, "Lf"'] \arrow[r, equals] & X \arrow[d, "f"] \\ Ef \arrow[ur, dashed, "r_f"] \arrow[r, "Rf"'] & Y
       \end{tikzcd}\]
       in the canonical commutative square against its left factor. As the choice of such a lift for $f$ induces one for any pullback, this defines a notion of fibered structure.
\end{ex}

\begin{ex} If $\FF_1 \to \EE$ and $\FF_2 \to \EE$ are notions of fibered structure, then so is the pullback $\FF_1 \times_\EE \FF_2 \to \EE$, where an $(\FF_1 \times_\EE \FF_2)$-structure on a morphism is just a pair given by an $\FF_1$-structure and an $\FF_2$-structure.
\end{ex}

\begin{ex} When $\cE$ is locally presentable and locally cartesian closed, there is a $\lambda$ so that for any $\kappa \triangleright \lambda$, the relatively $\kappa$-presentable morphisms define a full notion of fibered structure $\EE^\kappa\hookrightarrow \EE$.
\end{ex}

\subsection{Universes in model categories}

A suitable notion of fibered structure on a suitable model category will have a universe in the sense of the following definition.

\begin{defn}\label{defn:fib-universe} Let $\cE$ be a model category and $\FF$ a notion of fibered structure. A \textbf{universe} for $\FF$ is a cofibrant object $U$, together with an $\FF$-algebra $\pi \colon \tilde{U} \fto U$ so that the map  $\pi \colon \cE(-,U) \fwto \FF$ is an acyclic fibration, meaning that this map has the right lifting property against all cofibrations $i \colon A \rightarrowtail B$ in $\cE$
  \[
    \begin{tikzcd}
      \cE(-,A) \arrow[d, "i\cdot-"'] \arrow[r, "h"] & \cE(-,U) \arrow[d, tfibarrow] \\ \cE(-,B) \arrow[r, "p"'] \arrow[ur, dashed, "k"'] & \FF\end{tikzcd}\]
\end{defn}

Explicitly, this means that $\pi$ satisfies realignment for $\FF$-algebras and $\FF$-morphisms: given $\FF$-algebras $p$ and $q$ so that the solid-arrow squares displayed below are both $\FF$-morphisms, there exist arrows defining a third square that is a pullback and an $\FF$-morphism and makes the diagram commute:
\[
\begin{tikzcd}
  D \arrow[dd, two heads, "q"'] \arrow[rr] \arrow[drr, phantom, "\lrcorner" very near start] \arrow[dddr, phantom, "\lrcorner" very near start]\arrow[dr] & & \tilde{U} \arrow[dd, two heads, "\pi"] \\ & E  \arrow[dr, phantom, "\lrcorner" very near start]\arrow[ur, dashed] & ~ \\ A \arrow[dr, tail, "i"'] \arrow[rr, "h" near end] & & U \\ & B \arrow[ur, dashed, "k"'] \arrow[from=uu, two heads, crossing over, "p"' near start]
\end{tikzcd}  
\]

\begin{defn} A notion of fibered structure $\FF$ is \textbf{locally representable} if the strict discrete fibration $\phi \colon \FF \to \EE$ is representable, meaning that for any $Z \in \cE$ and pullback
  \[
    \begin{tikzcd} \PP \arrow[r] \arrow[d] \arrow[dr, phantom, "\lrcorner" very near start] & \FF \arrow[d, "\phi"] \\ \cE(-,Z) \arrow[r, "f"'] & \EE\end{tikzcd}\]
the object $\PP$ is isomorphic to a representable.
\end{defn}

\begin{ex}
For a full notion of fibered structure associated to a class of fibrations in a presheaf category, local representability is equivalent to the condition that a map is a fibration if and only if its pullbacks to representables are fibrations \cite[3.16]{shulman}. In particular, by Remark \ref{rmk:kan-local}, the Kan fibrations are locally representable as a full notion of fibered structure.
\end{ex}

These notions are connected in the following theorem.

\begin{thm}[{\cite[5.10]{shulman}}] In a Grothendieck 1-topos with a combinatorial model structure in which all cofibrations are monomorphisms, any small-groupoid-valued\footnote{Recall that any $\phi \colon \FF \to \EE$ must have small fibers: at most a set's worth of $\FF$-structures on a given $f \colon X \to Y$. This condition means that there is at most a set's worth of isomorphism classes of $\FF$-algebras with codomain $Y$. This might be achieved by replacing $\FF$ by $\FF^\kappa$ for some regular cardinal $\kappa$.} locally representable notion of fibered structure has a universe.
\end{thm}

The proof is by an adaptation of Quillen's small object argument that uses the hypotheses on the notion of fibered structure $\FF$ to ensure that it defines a \textbf{stack for cell complexes}, meaning $\FF \colon \cE^\op \to \GPD$ preserves coproducts, pushouts of cofibrations, and transfinite composites of cofibrations in the weak bicategorical sense.

It remains to connect a notion of fibered structure on a model category $\cE$ to the fibrations in that model category. For any notion of fibered structure $\FF$ let $|\FF|$ be the image of the map $\phi \colon \FF  \to \EE$. Thus $|\FF| \hookrightarrow \EE$ is a full notion of fibered structure --- which, however, is not  generally locally representable, even if $\FF$ is --- and the $|\FF|$-algebras are the morphisms that admit some $\FF$-structure.

\begin{defn} A notion of fibered structure $\FF$ is \textbf{relatively acyclic} $\FF \to |\FF|$ is an acyclic fibration, meaning for any cofibration $i \colon A \cto B$ there is a lift
  \[
    \begin{tikzcd} \cE(-,A) \arrow[d, "i\cdot -"'] \arrow[r, "q"] & \FF \arrow[d, tfibarrow] \\ \cE(-,B) \arrow[r, "p"'] \arrow[ur, dashed] & {|\FF|}
    \end{tikzcd}
  \]
  That is, for any pullback
  \[
    \begin{tikzcd}
      D \arrow[r, "e"] \arrow[d, "q"'] \arrow[dr, phantom, "\lrcorner" very near start] & E \arrow[d, "p"] \\ A \arrow[r, tail, "i"'] &  B     
    \end{tikzcd}
  \] with $q,p$ both $\FF$-algebras and $i$ a cofibration there exists a new $\FF$-algebra structure on $p$ making the square an $\FF$-morphism.
\end{defn}

In practice, $\FF$ will be locally representable, and thus will admit a universe $U$, while $|\FF|$, which will be the full notion of fibered structure defined by the fibrations, may not be. When $\FF$ is relatively acyclic, we can compose the acyclic fibrations 
\[
  \begin{tikzcd} \cE(-,U) \arrow[r, "\pi", utfibarrow] & \FF \arrow[r, utfibarrow] & {|\FF|}
  \end{tikzcd}
\]
and regard $U$ as a universe for the fibrations satisfying realignment.

\begin{thm} Let $\cE$ be a right proper simplicial Cisinski model category, and $\FF$ a locally representable, relatively acyclic notion of fibered structure for which $|\FF|$ is the class of fibrations. Then there is a regular cardinal $\lambda$ so that for any regular cardinal $\kappa \triangleright\lambda$, there exists a relatively $\kappa$-presentable fibration $\pi \colon \tilde{U} \to U$ so that
  \begin{enumerate}
    \item Every relatively $\kappa$-presentable fibration is a pullback of $\pi$.
    \item $U$ is fibrant.
    \item $\pi$ satisfies the univalence axiom.
  \end{enumerate}
\end{thm}
\begin{proof}
  The results proven above are applied to $\FF^\kappa \coloneqq \FF \times_\EE \EE^\kappa$, which remains locally representable and relatively acyclic. As with the simplicial model, the condition of Definition \ref{defn:fib-universe} is used to reduce the verification of fibrancy of $U$ and univalence of $\pi$ to the fibration extension property and equivalence extension property of $\cE$, which can be proven by applying previously-developed techniques.
\end{proof}

\subsection{First examples of type-theoretic model topoi}

Now that we have developed some understanding of notions of fibered structure, we belatedly present some examples of type-theoretic model topoi.

\begin{ex}\label{ex:ttmt-spaces} The category of simplicial sets with its Kan-Quillen model structure is a type theoretic model topos with $\FF$ the full notion of fibered structure defined by the Kan fibrations.
\end{ex}

\begin{prop} If $\cE$ is a type theoretic model topos so is $\cE_{/X}$.
\end{prop}
\begin{proof} All of the required structure is created by the forgetful functor $U \colon \cE_{/X} \to \cE$.
\end{proof}

\begin{prop}
Products of type-theoretic model toposes are also type-theoretic model toposes.
\end{prop}
\begin{proof} All of the required structure is defined pointwise.
\end{proof}

To demonstrate that all model toposes can be presented by type-theoretic model toposes, it remains to prove two more general closure theorems. We will show:
\begin{itemize}
  \item That if $\cE$ is a type theoretic model topos, then so is the injective model structure on $\cE^{\cD^\op}$ for any small simplicial category $\cD$.
  \item For any set $S$ of morphisms in a type-theoretic model topos $\cE$ so that $S$-localization is left exact, the localized model structure $\cE_{S}$ is again a type-theoretic model topos.
\end{itemize}

A key technical challenge to overcome in order to achieve these results is that the model categorical literature provided no explicit description of fibrations in the injective model structure, much less in localizations thereof. Shulman characterizes the injective fibrations as pointwise fibrations that are algebras for a certain pointed endofunctor called the \emph{cobar construction} introduced in \S\ref{ssec:cobar}. The corresponding problem for localizations is solved by showing that any left exact localization of an $\infty$-topos yields an \emph{internal left exact localization}, using results of \cite{ABFJ}, and then constructing an internal universe of local objects. This is the subject of \S\ref{ssec:localization}.

\subsection{Injective model structures and the cobar construction}\label{ssec:cobar}

Recall that the projective and injective model structures on a diagram category  $\cE^{\cD^\op}$ are lifted from a  model structure on $\cE^{\ob\cD}$ inherited pointwise from a model structure on $\cE$. In \S\ref{ssec:simp-presheaves}, we first attempted to define the cofibrations, weak equivalences, and fibrations in $\cE^{\cD^\op}$ as those simplicial natural transformations whose components lie in those classes in $\cE$. For a pointwise acyclic cofibration and a pointwise fibration it is possible to choose pointwise diagonal lifts for each object $d \in \cD$, but these do not assemble into a genuine natural transformation (since solutions to lifting problems are not unique). However, they do always form a \emph{homotopy coherent natural transformation}, due to the homotopical uniqueness of solutions to lifting problem:
\[
  \begin{tikzcd} A_d \arrow[d, "j_d"', tcofarrow] \arrow[r, "y_d"] & Y_d \arrow[d, "p_d", two heads] & \arrow[d, phantom, "\in\ \cE"] & & A \arrow[d, "j"', tcofarrow] \arrow[r, "y"] & Y \arrow[d, "p", two heads] & \arrow[d, phantom, "\in\ \cE^{\cD^\op}"]\\ X_d \arrow[ur, dashed, "\ell_d"'] \arrow[r, "b_d"'] & B_d &~ &  & X \arrow[r, "y"'] \arrow[ur, rightsquigarrow, dashed, "\ell"'] & B & ~
  \end{tikzcd}
\]
Thus, a strict lift could be obtained by somehow rectifying a homotopy coherent natural transformation to a strict one. And indeed, there is a natural bijection between homotopy coherent transformations $X \rightsquigarrow Y$ and strict natural transformations $B(L,UL,UX) \to Y$ whose domain is given by the \textbf{bar construction} defined below. A dual natural bijection classifies homotopy coherent natural transformations by strict natural transformations $X \to C(R, UR, UY)$ mapping into the \textbf{cobar construction}.

There are canonical maps $\epsilon_X \colon B(L,UL,UX) \to X$ and $\nu_Y \colon Y \to C(R,UR,UY)$ that form a pair of simplicial homotopy equivalences. If we have a strict transformation $s \colon X \to B(L,UL,UX)$ or $r \colon C(R,UR,UY) \to Y$ that respectively define a strict section and strict retraction of these natural maps then every homotopy coherent natural transformation with domain $X$ can be rectified to a strict one, and dually for homotopy coherent natural transformations with codomain $Y$. This gives a new characterization of the projective cofibrant objects $X$ and injective fibrant objects $Y$. The relative version of this construction will characterize the injective fibrations in an induced type-theoretic model topos structure on a diagram category.

\begin{dig} This is inspired by an analogous result in 2-category theory described in the dual case. In \cite{lack}, Lack characterizes the cofibrant objects in the projective model structure as those who admit a coalgebra structure for the dual pseudomorphism classifier. The 2-categorical cofibrant objects are called ``flexible'' so fibrant objects in the injective model structure might be called ``coflexible.''
\end{dig}

There is a slick definition of the bar and cobar constructions via the free adjunction of Schanuel and Street \cite{SchanuelStreet}. 

\begin{defn}[the free adjunction]\label{defn:free-adjunction} Let $\Adj$ denote the 2-category with two objects $+$ and $-$ and the four hom-categories 
  \[ \Adj(+,+) \coloneqq \DDelta_+\ ,\quad \Adj(-,-) \coloneqq \DDelta_+^\op , \quad \Adj(-,+) \coloneqq\DDelta_\top , \quad \Adj(+,-) \coloneqq \DDelta_\bot\]
  displayed in the following cartoon:
  \[
  \begin{tikzcd}[column sep=large]
  - \arrow[from=r, bend right, "\DDelta_\bot\cong\DDelta_\top^\op"'] \arrow[loop left, "\DDelta_+^\op"] \arrow[r, phantom, "\perp"] & + \arrow[loop right, "\DDelta_+"] \arrow[from=l, bend right, "\DDelta_\top\cong\DDelta_\bot^\op"']
  \end{tikzcd}
  \]
  Here $\DDelta_+$ is the category of finite ordinals and order preserving maps, $\DDelta$ is the full subcategory spanned by the non-empty ordinals, and $\DDelta_\top, \DDelta_\bot \subset \DDelta \subset \DDelta_+$ are the subcategories of order-preserving maps that preserve the top or bottom elements. Their intersection \[\DDelta_\bot \cap \DDelta_\top \cong \DDelta_+^\op\] is the subcategory of order-preserving maps that preserve both the top and bottom elements in each ordinal. This identifies $\DDelta_+^\op$ with the subcategory of ``intervals.'' 
\end{defn}  

\begin{defn}[co/simplicial co/bar construction] The ``freeness'' of the free adjunction has to do with the fact that an adjunction
\[ 
\begin{tikzcd} \cM \arrow[r, bend right, "U"'] \arrow[r, phantom, "\bot"] & \cN \arrow[l, bend right, "F"']
\end{tikzcd}
\]
is encoded by a unique 2-functor $\Adj \to \Cat$ defined by the data:
\[ \begin{tikzcd}[row sep=tiny] \Adj \arrow[r] & \Cat \\ - \arrow[r, mapsto] & \cM \\ + \arrow[r, mapsto] & \cN
\end{tikzcd} \qquad \begin{tikzcd}[row sep=tiny]\DDelta_+ \to \cN^\cN \\ \DDelta_+^\op \to \cM^\cM \\ \DDelta_\top \to \cN^\cM \\ \DDelta_\bot \to \cM^\cN
\end{tikzcd}
\]
The \textbf{two-sided simplicial bar construction} and \textbf{two-sided cosimplicial cobar construction} are the simplicial and cosimplicial objects defined by rearranging the action on homs on the comonad and monad side to define functors \begin{equation}\label{eq:bar-cobar}
  \begin{tikzcd}[column sep=huge]  \cM \arrow[r, "{B_\bullet(F,UF,U-)}"]&  \cM^{\DDelta^\op_+} & & \cN \arrow[r, "{C^\bullet(U,FU,F-)}"] & \cN^{\DDelta_+}
  \end{tikzcd}
\end{equation} while the actions of the other two homs define diagrams 
  \[
  \begin{tikzcd}[column sep=huge]  \cM \arrow[r, "{B_\bullet(UF,UF,U-)}"]&  \cN^{\DDelta_\top} & & \cN \arrow[r, "{C^\bullet(FU,FU,F-)}"] & \cM^{\DDelta_\bot}
  \end{tikzcd}
  \] 
\end{defn}

The interest of the latter two diagrams is that the presence of the co/aug\-ment\-a\-tions and ``extra co/degen\-er\-acies'' imply that the totalizations/geometric realizations of the underlying co/simplicial objects are given by the co/augmentations. 

\begin{defn}[bar and cobar construction]
  Suppose $\cM$ is simplicially enriched, tensored, and cocomplete. Then the \textbf{bar construction} is the  functor $B(F,UF,U-) \colon \cM \to \cM$ defined as the geometric realization of the underlying simplicial object of \eqref{eq:bar-cobar}. Dually, when $\cN$ is simplicially enriched and complete with cotensors, the \textbf{cobar construction} is the functor $C(U,FU,F-) \colon \cN \to \cN$ defined as the totalization of the underlying cosimplicial object of \eqref{eq:bar-cobar}.
\end{defn}

The former has a natural augmentation $\epsilon \colon B(F,UF,U-) \Rightarrow \id_\cM$ and the latter has a natural coaugmentation  $\nu \colon \id_\cN \Rightarrow C(U,FU,F-)$ and both are carried to simplicial homotopy equivalences by the functors $U$ and $F$ respectively. Thus their components are carried to maps that are weak equivalences when $\cN$ or $\cM$ are simplicial model categories, respectively. This leads to our desired characterization of the injective fibrations.

\begin{thm}\label{thm:inj-fib} Let $U \colon \cM \to \cN$ be a simplicial functor with both simplicial adjoints $L \dashv U \dashv R$ where $\cN$ is a simplicial model category whose cofibrations are the monomorphisms. Under a technical condition\footnote{Shulman asks that the unit of the monad is Quillen-cofibrant or equivalently that the counit of the comonad is Quillen-fibrant.}, then $f \colon X \to Y$ in $\cM$ is an injective fibration if and only if it is a $U$-fibration and the left map $\lambda_f$ of the \textbf{cobar factorization} $f = \rho_f \cdot \lambda_f$ defined below has a retraction over $Y$.
\[
\begin{tikzcd} X \arrow[drr, bend left, "\nu_X"] \arrow[ddr, bend right, "f"'] \arrow[dr, dashed, "\lambda_f" description] \\ & Ef \arrow[r, dotted] \arrow[d, dotted, "\rho_f"] \arrow[dr, phantom, "\lrcorner" very near start] & C(R,UR,UX) \arrow[d, "{C(R,UR,Uf)}"] \\  & Y \arrow[r, "\nu_Y"'] & C(R,UR,UY)
\end{tikzcd}
\]
In particular $X$ is injectively fibrant if and only if it is $U$-fibrant and the map $\nu_V \colon X \to C(R,UR,UX)$ has a retraction.
\end{thm}
\begin{proof} If $f$ is a $U$-fibration then $\rho_f$ is a pullback of an injective fibration and hence an injective fibration. Thus if $f$ is a retract of it, it is an injective fibration.

Conversely if $f$ is an injective fibration, then $f$ is a $U$-fibration. The technical hypothesis implies that $\lambda_f$ is a $U$-trivial cofibration so $f$  $\lambda_f$, yielding the desired retraction over $Y$.
\end{proof}

\begin{thm} Let $\cE$ be a type-theoretic model topos and $T$ a simplicially enriched Quillen unit-cofibrant monad on $\cE$ having a simplicial right adjoint $S$. Then the category $\cE^T$ of $T$-algebras is again a type-theoretic model topos.
\end{thm}
\begin{proof} There is an induced simplicial comonad structure on $S$ so that $T$-algebras coincide with $S$-coalgebras. In particular, $U \colon \cE^T \to \cE$ has both simplicial adjoints, and thus $\cE^T$ has an injective model structure by Proposition \ref{prop:proj-inj}. Moreover, by classical theorems, $\cE^T$ is a Grothendieck topos and simplicially locally cartesian closed (these facts rely more on the comonadicity than the monadicity).

If $\FF$ is the locally representable and relatively acyclic notion of fibered structure for the fibrations on $\cE$ let $\FF^T = U^{-1}(\FF) \times_\EE \RR_E$ where $\RR_E$ is the notion of fibered structure of Example \ref{ex:fun-fact} for the cobar functorial factorization. This notion of fibered structure is locally representable and relatively acyclic. Theorem \ref{thm:inj-fib} proves that $|\FF^T|$-algebras are exactly the injective fibrations.
\end{proof}

\begin{cor}\label{cor:ttmt-diagrams} If $\cE$ is a type-theoretic model topos and $\cD$ is a small simplicially enriched category, then the injective model structure on $\cE^{\cD^\op}$ is a type-theoretic model topos.
\end{cor}
\begin{proof} The monad in $\cE^{\ob\cD}$ whose category of algebras is $\cE^{\cD^\op}$ is Quillen unit-cofibrant---essentially because the hom-categories in $\cD$ are cofibrant and the inclusion of the identity is a cofibration---and has a simplicial right adjoint. 
\end{proof}

\subsection{Localizations}\label{ssec:localization}

The final required result is:

\begin{thm}\label{thm:ttml-localization} For any set $S$ of morphisms in a type-theoretic model topos $\cE$ such that $S$-localization is left exact, the localized model structure $\cE_{S}$ is again a type-theoretic model topos.
\end{thm}
\begin{proof}
Since $\cE_S$ has the same underlying category and cofibrations, it is a simplicially locally cartesian closed Grothendieck 1-topos with cofibrations being the monomorphisms. It is combinatorial and simplicial and right proper since the localization is left exact. 

We may replace $S$ by weakly equivalent fibrations between fibrant objects without changing the localization. Then we claim that there is a locally representable, relatively acyclic notion of fibered structure $\LL_S$ such that the $\LL_S$-algebras are the correct fibrations. 
\end{proof}

The construction of $\LL_S$, which we don't give here, relies on a characterization of left exact localizations from \cite{ABFJ} that involves the following construction. Given a fibration $f \colon A \fto B$ between fibrant objects, the \textbf{fibrant diagonal} is the fibrant replacement $\Delta f \colon A' \fto A \times_B A$ of the strict diagonal. Write $\Delta^nf$ for the $n$-fold iterate with $\Delta^0f = f$, write $S^\Delta$ for the class of iterated fibrant diagonals of morphisms in $S$, and write $(S^\Delta)^\textup{pb}$ for all pullbacks of these morphisms. Anel, Biedermann, Finster, and Joyal chacterize left exact localizations as follws: 

\begin{thm}[{\cite{ABFJ}}] For any set $S$ of fibrations between fibrant objects:
\begin{enumerate}
\item $(S^\Delta)^\textup{pb}$-localization is left exact.
\item If $S$-localization is left exact it coincides with $(S^\Delta)^\textup{pb}$-localization.
\end{enumerate}
\end{thm}

\begin{cor} Any $\infty$-topos  can be presented by a model category that interprets Martin-L\"{o}f type theory with the structure listed above.
\end{cor}
\begin{proof}
Every $\infty$-topos can be presented by a model topos, and every model topos is Quillen equivalent to a left exact localization of an injective model structure on simplicial presheaves. By Example \ref{ex:ttmt-spaces}, Corollary \ref{cor:ttmt-diagrams}, and Theorem \ref{thm:ttml-localization} such model categories are type-theoretic model toposes. Therefore, any Grothendieck $\infty$-topos can be presented by a type-theoretic model topos.
\end{proof}

\end{document}